\documentclass[11pt]{amsart}

\usepackage{amsmath,amssymb,amsthm,amsfonts}
\usepackage{mathrsfs}
\usepackage{enumitem}
\usepackage[colorlinks=true,citecolor=blue,linkcolor=blue,urlcolor=blue]{hyperref}
\usepackage{longtable}
\usepackage{booktabs}   
\newtheorem{theorem}{Theorem}[section]

\newtheorem{proposition}[theorem]{Proposition}
\newtheorem{lemma}[theorem]{Lemma}
\newtheorem{corollary}[theorem]{Corollary}

\theoremstyle{definition}
\newtheorem{definition}[theorem]{Definition}
\newtheorem{remark}[theorem]{Remark}

\numberwithin{equation}{section}

\title{Necessary Conditions for Single-Critical-Point\\ Higher-Order Szeg\H{o} Sum Rules in OPUC}
\author{Daxiong Piao}

\address{School of Mathematical Sciences, Ocean University of China, Qingdao 266100, China}
\email{dxpiao@ouc.edu.cn}
\date{\today}
\begin{document}

\begin{abstract}
We prove the necessity part of the higher-order Szeg\H{o} theorem on the
unit circle for the single-critical-point weights
$H_m(e^{i\theta})=(1-\cos\theta)^m$, $m\ge1$.  If
$\{\alpha_n\}_{n\ge0}$ are the Verblunsky coefficients of a nontrivial
probability measure
$d\mu=w(\theta)d\theta/(2\pi)+d\mu_{\mathrm s}$, then the weighted
Szeg\H{o} condition
$\int_0^{2\pi}
        (1-\cos\theta)^m\log w(\theta)\frac{d\theta}{2\pi}>-\infty$
implies
$\Delta^m\alpha\in\ell^2,
        \,\,
        \alpha\in\ell^{2m+2}.$

The proof uses a finite-volume version of Yan's higher-order sum rule.  The
quadratic part yields the $m$-th difference energy, and the logarithmic
tail yields the $\ell^{2m+2}$-control.  The non-sign-definite critical
terms are treated in two steps.  First, the quartic principal critical block
is isolated using the Yan quotient-algebra normal representative and shown
to have a positive semidefinite Gram representation.  Second, the remaining
non-principal critical terms are controlled by the diagonal-vanishing
property
$\mathcal Y_{k,\mathrm{crit}}^{(m)}
        \in
        \mathfrak I_k^{\,m+1-k},
        \,\, 2\le k\le m,$
together with the Breuer--Simon--Zeitouni normal form, discrete
interpolation, and Young's inequality.  These estimates yield a uniform
finite-volume coercive bound, from which the necessity theorem follows for
all $m\ge1$.
\end{abstract}

\maketitle

\tableofcontents

\section{Introduction}
\label{sec:introduction}

Let $\mu$ be a nontrivial probability measure on the unit circle
$\partial\mathbb D$, and write its Lebesgue decomposition as
$d\mu = w(\theta)\frac{d\theta}{2\pi}  +  d\mu_{\mathrm s}.$
Let
$\{\alpha_n\}_{n=0}^{\infty}$
be the corresponding Verblunsky coefficients. The classical Szeg\H{o}
theorem relates the integrability of $\log w$ to square summability of the
Verblunsky coefficients. More precisely,
$\int_0^{2\pi}
        \log w(\theta)\frac{d\theta}{2\pi}
        >
        -\infty
        \,\,\Longleftrightarrow\,\,
        \sum_{n=0}^{\infty}|\alpha_n|^2<\infty.$
This  is one of the central theorems in the theory of
orthogonal polynomials on the unit circle(OPUC); see Szeg\H{o}'s classical book
and Simon's monographs
\cite{Szego1939,SimonOPUC1,SimonOPUC2,SimonSzego}.

Higher-order Szeg\H{o} theorems refine the classical theorem by inserting
weights that vanish at prescribed points of the unit circle. Such results
are naturally formulated as sum rules: a weighted logarithmic integral on
the spectral side is related to summability and difference conditions on the
Verblunsky coefficients. In this paper we focus on the single-critical-point
weight
$H_m(e^{i\theta})
        =
        (1-\cos\theta)^m,
        \,\, m\ge1.$
Equivalently, with $z=e^{i\theta}$,
$H_m(z)
        =
        2^{-m}(1-z)^m(1-z^{-1})^m,$
so that $H_m$ has a zero of order $2m$ at the point $z=1$. Throughout the paper we use the nonnegative sign convention
$$\mathcal K_m(\mu)
        :=
        \int_0^{2\pi}
        (1-\cos\theta)^m
        \log\frac1{w(\theta)}
        \frac{d\theta}{2\pi}.$$
Thus $\mathcal K_m(\mu)<\infty$ is the weighted Szeg\H{o} condition
$$\int_0^{2\pi}
        (1-\cos\theta)^m
        \log w(\theta)\frac{d\theta}{2\pi}
        >
        -\infty.$$
The coefficient-side conditions predicted by the higher-order theory are
$\Delta^m\alpha\in\ell^2,
        \,\,
        \alpha\in\ell^{2m+2},$
where $\Delta\alpha_n=\alpha_{n+1}-\alpha_n$.

There is a substantial literature on higher-order Szeg\H{o} sum rules and
on the coefficient-side conditions associated with them. Simon proposed
natural coefficient-side conditions for polynomial weights with finitely many
critical points; see also the discussion in Simon's books
\cite{SimonOPUC2,SimonSzego}. Simon and Zlato\u{s} verified these
conditions in several low-order cases, including the case of two singular
points \cite{SimonZlatos2005}. Golinskii and Zlato\u
{s} proved further
important cases of higher-order Szeg\H{o} theorems, including results under
additional a priori summability assumptions such as $\alpha\in\ell^4$
\cite{GZ2007}. Related asymptotic questions for orthogonal polynomials in
Szeg\H{o} classes with polynomial weights were also studied by Denisov and
Kupin \cite{DenisovKupin2006}.

Lukic later showed that Simon's original coefficient-side conditions are not
correct in full generality and formulated a replacement set of conditions,
now known as Lukic's conditions \cite{Lukic2014}. In a convenient form, if
the critical points are
$e^{i\theta_1},\ldots,e^{i\theta_K}$
with multiplicities
$m_1,\ldots,m_K,$
then Lukic's conditions involve the global difference condition
$\prod_{j=1}^K (S-e^{-i\theta_j})^{m_j}\alpha\in\ell^2$
together with local higher-integrability conditions near each critical
point. For a single critical point, these conditions reduce precisely to
$(S-e^{-i\theta_1})^{m_1}\alpha\in\ell^2,
        \,\,
        \alpha\in\ell^{2m_1+2}.$
Lukic also proved supporting results in the one-critical-point case; in
particular, under the assumption
$(S-e^{-i\theta_1})\alpha\in\ell^2,$
he showed that the corresponding higher-order Szeg\H{o} condition is
equivalent to
$\alpha\in\ell^{2m_1+2}$
\cite{Lukic2016}. Thus, for the one-point weight considered in the present
paper, Lukic's conditions are exactly
$\Delta^m\alpha\in\ell^2,
        \,\,
        \alpha\in\ell^{2m+2}.$

Another important line of work connects sum rules with large deviation
principles. Large-deviation methods for spectral measures and sum rules were
developed in several settings by Gamboa, Nagel, and Rouault
\cite{GamboaNagelRouault2016,GamboaNagelRouault2017,GamboaNagelRouaultMatrix}.
This approach is part of a broader circle of ideas relating random matrix
theory, spectral measures, and large deviations
\cite{AndersonGuionnetZeitouni2010,DemboZeitouni1998,DeuschelStroock1989}.
Breuer, Simon, and Zeitouni(BSZ) developed this perspective further for
higher-order OPUC sum rules and the Lukic conjecture
\cite{BSZ-Duke2018,BSZ-pedagogical2018}. They derived general finite-volume
sum-rule identities for polynomial weights and showed that the main
difficulty is not merely the existence of a sum rule, but rather the
analysis of the complicated coefficient-side expressions in terms of
Verblunsky coefficients. They also verified Lukic's conjecture in several
important cases, including cases in which Simon's original conditions and
Lukic's replacement conditions differ.

Analogous questions for Jacobi matrices have also played an important role
in the development of sum-rule methods. The Killip--Simon theorem is the
central example \cite{KillipSimon2003}, and generalized sum rules and
related coefficient conditions were studied in
\cite{NazarovPeherstorferVolbergYuditskii2005,Kupin2004,LaptevNabokoSafronov2003}.
Although the present paper concerns OPUC, this literature provides an
important parallel perspective on the relation between spectral integrals
and coefficient-side coercivity.

We also mention related probabilistic developments.  Gamboa, Nagel
and Rouault derived sum rules and spectral gems for polynomial potentials
on the unit circle by comparing spectral and coefficient large deviation
principles, including ungapped, gapped, and multi-cut regimes
\cite{GamboaNagelRouault2025}.  The present paper is independent of this
large-deviation approach and instead uses the finite-volume algebraic
formalism of Yan together with the BSZ normal-form reduction.

Building on the general sum-rule framework of BSZ and
motivated by the difficulty of analyzing the coefficient-side expressions
directly, Yan introduced an algebraic model for higher-order sum rules on
the unit circle \cite{Yan2018}. Yan's model relates the finite-volume
sum-rule expressions to explicit algebraic polynomials, including
Hall--Littlewood type polynomials; for background on Hall--Littlewood
polynomials, see Macdonald's book \cite{Macdonald1998}. This algebraic
model provides a systematic way to organize the homogeneous terms on the
Verblunsky side. The algebraic finite-volume sum-rule formalism used in the
present paper is based on this approach.

A key precursor to the present work is Yan's sufficiency theorem. In
\cite{Yan2018}, Yan used the algebraic model to prove the sufficiency part
of Lukic's conjecture in the case of a single critical point with arbitrary
multiplicity. After translating Yan's normalization to the convention used here, his
Theorem~1.14 implies that, for
$H_m(e^{i\theta})=(1-\cos\theta)^m$,
$\Delta^m\alpha\in\ell^2,
        \,\,
        \alpha\in\ell^{2m+2}
        \,\,\Longrightarrow\,\,
        \mathcal K_m(\mu)<\infty.$
Thus Yan's theorem gives the
sufficiency direction in the single-critical-point, arbitrary-order case.

The purpose of the present paper is to prove the converse implication. Our
main theorem shows that finiteness of the order-$m$ spectral side forces
$\Delta^m\alpha\in\ell^2,
        \,\,
        \alpha\in\ell^{2m+2}.$
Equivalently,
$\sum_{n=0}^{\infty}|\Delta^m\alpha_n|^2<\infty,
        \,\,
        \sum_{n=0}^{\infty}|\alpha_n|^{2m+2}<\infty.$
Combining this necessity theorem with Yan's sufficiency theorem
\cite{Yan2018} yields the full necessary-and-sufficient characterization
$\mathcal K_m(\mu)<\infty
        \,\,\Longleftrightarrow\,\,
        \Delta^m\alpha\in\ell^2
        \ \text{ and }\
        \alpha\in\ell^{2m+2},$
for every $m\ge1$ in the single-critical-point case.

This result also realizes a possibility explicitly suggested by Yan. At the
end of the introduction of \cite{Yan2018}, Yan observed that, although his
paper focused on the sufficiency of Lukic's conditions, the algebra model
also had the potential to prove necessity. Roughly speaking, Yan pointed out
that if one could decompose the algebraic degree $2k$ components into terms
whose images under the coefficient map are summable, and if that summability
were equivalent to Lukic's conditions, then the necessity of Lukic's
conditions would follow. The present paper confirms this perspective in the
single-critical-point case. The additional ingredient needed here is that
Yan's algebraic structure must be combined with the BSZ normal form method \cite{BSZ-Duke2018,BSZ-pedagogical2018} and a discrete
Gagliardo--Nirenberg interpolation argument, in the spirit of the classical
interpolation inequalities of Gagliardo and Nirenberg
\cite{Gagliardo1958,Nirenberg1959}.

We now describe the mechanism of the proof. Yan's Hall--Littlewood type
representation identifies the homogeneous degree $2k$ pieces of the
finite-volume sum rule with explicit algebraic polynomials. For the
single-critical-point weights $H_m$, we prove that the critical degree $2k$ component
vanishes to order at least $m+1-k$ on the diagonal:
$\mathcal Y_{k,\mathrm{crit}}^{(m)}
        \in
        \mathfrak I_k^{\,m+1-k},
        \,\, 2\le k\le m.$
 Here $\mathfrak I_k$ denotes the ideal of polynomials vanishing at the
distinguished diagonal point where all shift variables are equal to $1$.
This diagonal vanishing provides the algebraic source of the required
difference count. After applying the coefficient map and passing to the
BSZ normal form, the corresponding local expressions
contain the required discrete differences.  There is, however, one quartic
principal critical block which must be treated separately.  The raw quartic
representative does not in general have a positive semidefinite Gram
representative.  We therefore first replace it by the Yan quotient-algebra
normal representative; its lowest critical homogeneous part gives an
explicit positive semidefinite quartic block.  After this block is separated,
the remaining non-principal critical terms are absorbable by
$\sum_n |\Delta^m\alpha_n|^2
        \,\,\text{and}\,\,
        \sum_n |\alpha_n|^{2m+2}.$
This positive-semidefinite-block plus absorption mechanism is the key step
in deriving the coercive finite-volume estimate.

Let us explain the two positive quantities. The quadratic part of the coefficient-side expansion produces the
difference energy. Since
$1-\cos\theta
        =
        \frac12|1-e^{i\theta}|^2,$
we have
$H_m(e^{i\theta})
        =
        2^{-m}|1-e^{i\theta}|^{2m}.$
Thus, at the level of Fourier symbols, the quadratic contribution is
associated with
$2^{-m}(1-P)^m(1-P^{-1})^m,$
where $P$ denotes the shift operator. As a quadratic form this gives
$2^{-m}\sum_n|\Delta^m\alpha_n|^2,$
up to finite-section boundary terms. This is the coercive source of the
condition $\Delta^m\alpha\in\ell^2$.

The second positive term comes from the logarithmic part of the finite-volume
sum rule. With the normalization used below, the degree $2k$ contribution in Yan's
homogeneous expansion contains a universal constant part equal to $-1/k$.
 After applying the
coefficient map, this contributes
$-\frac1k|\alpha_n|^{2k}.$
Therefore the logarithmic term combines with these constants as
$\log\frac1{1-|\alpha_n|^2}
        -
        \sum_{k=1}^{m}\frac{|\alpha_n|^{2k}}{k}.$
Since $0\le |\alpha_n|<1$, the Taylor expansion of
$\log(1/(1-x))$ gives
$$\log\frac1{1-|\alpha_n|^2}
        -
        \sum_{k=1}^{m}\frac{|\alpha_n|^{2k}}{k}
        =
        \sum_{j=m+1}^{\infty}
        \frac{|\alpha_n|^{2j}}{j}
        \ge
        \frac1{m+1}|\alpha_n|^{2m+2}.$$
This is the coercive source of the condition
$\alpha\in\ell^{2m+2}$.

The main difficulty is therefore the control of the remaining critical
terms.  These terms are not manifestly sign-definite.  In low orders they
can be handled by explicit calculation, but such computations quickly become
impractical as $m$ increases.  The contribution of this paper is to
replace the low-order cancellations by a structural argument.  The quartic
principal critical obstruction is handled by first passing to the Yan
quotient-algebra normal representative; the resulting lowest critical
homogeneous part admits an explicit positive semidefinite Gram
representation.  After this positive semidefinite block is separated, the
remaining non-principal critical terms are treated by the
BSZ normal form argument. More precisely, the diagonal-vanishing statement above implies, through the
BSZ normal form, that every remaining non-principal
critical monomial of degree $2k$, $2\le k\le m$, can be written, modulo
bounded telescoping and finite-section boundary terms, as a finite linear
combination of expressions of the form
$\prod_{\nu=1}^{2k}
        \bigl(\Delta^{r_\nu}\gamma_\nu\bigr)_{n+\ell_\nu},
        \,\,
        \gamma_\nu\in\{\alpha,\overline{\alpha}\},$
with bounded shifts $\ell_\nu$ and with total difference count satisfying
$\sum_{\nu=1}^{2k} r_\nu
        \ge
        m+1-k.$
This normal form replaces the explicit cancellations visible in the low
orders $m=1,2,3$.

The analytic input is an infinitesimal interpolation estimate for these
normal-form monomials. For every $\varepsilon>0$, each admissible
higher-order monomial $T_n$ satisfies
$\sum_{n=0}^{N}|T_n|
        \le
        \varepsilon
        \sum_{n=0}^{N+L_m}|\Delta^m\alpha_n|^2
        +
        \varepsilon
        \sum_{n=0}^{N+L_m}|\alpha_n|^{2m+2}
        +
        C_{\varepsilon,m}.$
Thus all remaining non-principal critical terms are infinitesimally
form-bounded with respect to the two coercive quantities.  The quartic
principal block is not absorbed; instead it is kept as a nonnegative
contribution after passing to the Yan quotient-algebra representative.  The
infinitesimal nature of the absorption estimate is essential.  Since the
logarithmic tail carries a fixed positive coefficient, the non-principal
critical terms must be absorbed with an arbitrarily small relative constant.

Combining the quadratic coercivity, the logarithmic tail, the positive
semidefinite quartic principal block, and the absorption of the remaining
non-principal critical terms yields the finite-volume coercive estimate
$$\mathcal K_m^{(N)}(\mu)
        \ge
        c_m
        \sum_{n=0}^{N}|\Delta^m\alpha_n|^2
        +
        c_m
        \sum_{n=0}^{N}|\alpha_n|^{2m+2}
        -
        C_m,$$
where $c_m>0$ and $C_m<\infty$ are independent of $N$.  Here
$\mathcal K_m^{(N)}(\mu)$ denotes the finite-volume spectral functional in
Yan's truncated sum rule, with the sign convention used in this paper. In Yan's finite-volume formulation, finiteness of
$\mathcal K_m(\mu)$ gives a uniform upper bound on the truncated spectral
quantities $\mathcal K_m^{(N)}(\mu)$. The coercive estimate
therefore gives
$$\sup_N
        \sum_{n=0}^{N}|\Delta^m\alpha_n|^2<\infty,
        \,\,
        \sup_N
        \sum_{n=0}^{N}|\alpha_n|^{2m+2}<\infty.$$
Passing to the limit $N\to\infty$ yields the desired coefficient-side
conditions.

We now state the main theorem.

\begin{theorem}[Arbitrary-\texorpdfstring{$m$}{m} necessity theorem]
\label{thm:intro-main-necessity}
Let $m\ge1$, and let $\mu$ be a nontrivial probability measure on
$\partial\mathbb D$ with Lebesgue decomposition
$d\mu
        =
        w(\theta)\frac{d\theta}{2\pi}
        +
        d\mu_{\mathrm s}.$
Let $\{\alpha_n\}_{n\ge0}$ be its Verblunsky coefficients. If
\begin{equation}
\label{eq:intro-weighted-szego-condition}
        \int_0^{2\pi}
        (1-\cos\theta)^m
        \log w(\theta)\frac{d\theta}{2\pi}
        >
        -\infty \,\, i.e. \,\, \mathcal K_m(\mu)< \infty,
\end{equation}
then
$\Delta^m\alpha\in\ell^2,
        \,\,
        \alpha\in\ell^{2m+2}.$
\end{theorem}

Combining Theorem~\ref{thm:intro-main-necessity} with Yan's sufficiency
theorem \cite[Theorem 1.14]{Yan2018} gives the following full equivalence.

\begin{corollary}[Single-critical-point equivalence]
\label{cor:intro-equivalence}
Under the same assumption of  Theorem~\ref{thm:intro-main-necessity}, \eqref{eq:intro-weighted-szego-condition} holds if and only if
$\Delta^m\alpha\in\ell^2
        \,\,\text{and}\,\,
        \alpha\in\ell^{2m+2}.$
\end{corollary}

The organization of the paper is as follows. In
Section~\ref{sec:yan-finite-volume} we recall Yan's algebraic model and the
finite-volume form of the higher-order sum rule. In
Section~\ref{sec:remainder-telescoping} we introduce the local remainder
classes and the telescoping calculus used to control finite-section boundary
terms. In Section~\ref{sec:constant-log-tail} we isolate the constant part
of Yan's homogeneous expressions and prove the positivity of the logarithmic
tail. Section~\ref{sec:quadratic-coercivity} identifies the quadratic
critical term and proves coercivity of the $m$-th difference energy. In
Section~\ref{sec:recursive-diagonal-vanishing} we analyze the recursive
structure of the algebraic expressions and prove the diagonal vanishing of
the critical part. Section~\ref{sec:bsz-normal-form} converts this diagonal
vanishing into BSZ normal form.
In Section~\ref{sec:absorption-coercive-estimate} we isolate the positive
semidefinite quartic principal block, prove infinitesimal absorption
estimates for the remaining non-principal critical terms, and derive the
coercive finite-volume estimate. The proof of the arbitrary-$m$ necessity theorem is
given in Section~\ref{sec:proof-necessity-arbitrary-m}. For the reader's convenience, a notation index is provided at the end of the paper.

\section{Yan's algebraic model and finite-volume sum rules}
\label{sec:yan-finite-volume}

This section recalls the finite-volume form of the higher-order OPUC sum
rules used throughout the paper. The general higher-order sum-rule framework
is due to Breuer--Simon--Zeitouni
\cite{BSZ-Duke2018,BSZ-pedagogical2018}. The algebraic finite-volume
formalism used here is based on Yan's model \cite{Yan2018}; we recall only
the components needed later and specialize them to the single-critical-point
weights
\begin{equation}
\label{eq:yan-weight-Hm-theta}
        H_m(e^{i\theta})
        =
        (1-\cos\theta)^m,
        \,\, m\ge1.
\end{equation}
Our normalization is
\begin{equation}
\label{eq:yan-normalization-one-minus-cos}
        1-\cos\theta
        =
        \frac12 |1-e^{i\theta}|^2.
\end{equation}
Equivalently, with $z=e^{i\theta}$,
\begin{equation}
\label{eq:yan-weight-Hm-z}
        H_m(z)
        =
        2^{-m}(1-z)^m(1-z^{-1})^m.
\end{equation}

Let $\mu$ be a nontrivial probability measure on $\partial\mathbb D$, and
write its Lebesgue decomposition as
$ d\mu =  w(\theta)\frac{d\theta}{2\pi} + d\mu_{\mathrm s}(e^{i\theta}).$
Let $  \{\alpha_n\}_{n=0}^{\infty}\subset\mathbb D $ be the corresponding Verblunsky coefficients. We write $P$ for the forward
shift,$ (P\alpha)_n=\alpha_{n+1}$, and set
\begin{equation}
\label{eq:yan-delta-definition}
        \Delta=P-1,
        \,\,
        (\Delta\alpha)_n=\alpha_{n+1}-\alpha_n.
\end{equation}
Thus
\begin{equation}
\label{eq:yan-mth-difference}
        \Delta^m\alpha=(P-1)^m\alpha.
\end{equation}

Recall that $\mathcal K_m(\mu)$ denotes the infinite-volume weighted
Szeg\H{o} functional associated with $H_m$, as defined in the Introduction.
Here we recall the corresponding finite-volume quantities
$\mathcal K_m^{(N)}(\mu)$ that appear in Yan's truncated sum rule. All
finite-volume quantities below are taken with the same sign convention as
$\mathcal K_m(\mu)$. The only spectral consequence used later is the
one-sided bound
\begin{equation}
\label{eq:yan-finite-volume-boundedness-assumption}
        \sup_{N\ge0}
        \mathcal K_m^{(N)}(\mu)
        <
        \infty,
\end{equation}
which follows from the finiteness of $\mathcal K_m(\mu)$.

\subsection{Local expressions and the coefficient map}
\label{subsec:yan-local-expressions}

Yan's algebraic model represents the coefficient side of the finite-volume
sum rule by local homogeneous expressions in the Verblunsky coefficients. We
recall the notation in a form adapted to the present argument.

A local balanced monomial of degree $2k$ is an expression of the form
\begin{equation}
\label{eq:yan-local-balanced-monomial}
        \prod_{\nu=1}^{k}\alpha_{n+i_\nu}
        \prod_{\mu=1}^{k}\overline{\alpha}_{n+j_\mu},
\end{equation}
where the shifts $i_\nu,j_\mu$ belong to a finite set whose size and range
are bounded in terms of $m$ only. Equivalently, using the shift $P$, one may encode such monomials by
formal shift variables acting on $\alpha_n$ and
$\overline{\alpha}_n$. The precise indexing convention is not important for
the estimates in later sections; what matters is that, for fixed $m$, only
finitely many shifts occur.

In Yan's notation, the algebraic $2k$-homogeneous expressions are
polynomials in shift variables
$x_1,\ldots,x_k,$ $y_1,\ldots,y_k,$
where the $x$-variables act on the $k$ unbarred factors and the
$y$-variables act on the $k$ barred factors. The coefficient map
$\phi_{2k}$ sends a monomial
$x_1^{i_1}\cdots x_k^{i_k}y_1^{j_1}\cdots y_k^{j_k}$
to
$\prod_{\nu=1}^{k}\alpha_{n+i_\nu}
        \prod_{\mu=1}^{k}\overline{\alpha}_{n+j_\mu}.$
 Thus, if
$\mathcal Y_k^{(m)}$ denotes the algebraic $2k$-homogeneous expression
associated with the weight $H_m$, then $ \bigl[\phi_{2k}(\mathcal Y_k^{(m)})\bigr]_n$ is a finite linear combination of monomials of the form
\eqref{eq:yan-local-balanced-monomial}. The expressions are balanced: every $2k$-homogeneous term contains $k$ factors of $\alpha$ and $k$
factors of $\overline{\alpha}$.

For fixed $m$, Yan's non-logarithmic coefficient expression has homogeneous
pieces of degrees
$2,4,\ldots,2m.$
We write
\begin{equation}
\label{eq:yan-homogeneous-decomposition}
        A_{m,n}
        =
        \sum_{k=1}^{m}
        A_{m,n}^{(2k)},
\end{equation}
where
\begin{equation}
\label{eq:yan-homogeneous-piece-definition}
        A_{m,n}^{(2k)}
        :=
        \bigl[\phi_{2k}(\mathcal Y_k^{(m)})\bigr]_n.
\end{equation}
The quadratic term $A_{m,n}^{(2)}$ will later produce the positive
$m$-th difference energy, while the higher terms
$A_{m,n}^{(4)},\ldots,A_{m,n}^{(2m)}$
will be converted into normal form and absorbed.

\subsection{The finite-volume sum rule}
\label{subsec:yan-finite-volume-sum-rule}

The following is the finite-volume form of Yan's sum rule  that will be used
in this paper.

\begin{theorem}[Yan finite-volume sum rule, single-critical-point case]
\label{thm:yan-finite-volume-sum-rule}
Let $m\ge1$, and let
$H_m(e^{i\theta})=(1-\cos\theta)^m.$
There exist finite-volume spectral functionals
$\mathcal K_m^{(N)}(\mu),
        \,\, N\ge0,$
and finite-section boundary terms
$\mathcal B_m^{(N)}(\alpha),$
such that
\begin{equation}
\label{eq:yan-finite-volume-sum-rule}
        \mathcal K_m^{(N)}(\mu)
        =
        \sum_{n=0}^{N}
        \log\frac1{1-|\alpha_n|^2}
        +
        \sum_{k=1}^{m}
        \sum_{n=0}^{N}
        \bigl[\phi_{2k}(\mathcal Y_k^{(m)})\bigr]_n
        +
        \mathcal B_m^{(N)}(\alpha).
\end{equation}
Moreover, the boundary terms are uniformly bounded:
\begin{equation}
\label{eq:yan-boundary-uniform}
        \sup_{N\ge0}
        |\mathcal B_m^{(N)}(\alpha)|
        \le
        C_m .
\end{equation}
Here $C_m<\infty$ is independent of $N$. This follows from the fact that
$\mathcal B_m^{(N)}$ contains only finitely many endpoint local monomials,
with shifts bounded in terms of $m$, and $|\alpha_n|\le1$.

Finally, finiteness of the infinite-volume spectral side implies the
one-sided finite-volume bound
\begin{equation}
\label{eq:yan-finite-volume-upper-bound}
        \sup_{N\ge0}
        \mathcal K_m^{(N)}(\mu)
        <
        \infty.
\end{equation}
\end{theorem}

This theorem is the finite-volume consequence of the higher-order OPUC sum
rules of Breuer--Simon--Zeitouni \cite{BSZ-Duke2018,BSZ-pedagogical2018}, written in
the algebraic form developed by Yan \cite{Yan2018}.  We shall not use the explicit spectral definition of
$\mathcal K_m^{(N)}(\mu)$; only the finite-volume identity
\eqref{eq:yan-finite-volume-sum-rule} and the uniform one-sided bound
\eqref{eq:yan-finite-volume-upper-bound} enter the proof.
Using the homogeneous decomposition
\eqref{eq:yan-homogeneous-decomposition}, define
\begin{equation}
\label{eq:A-mn-definition}
        A_{m,n}
        :=
        \sum_{k=1}^{m}
        \bigl[
        \phi_{2k}(\mathcal Y_k^{(m)})
        \bigr]_n .
\end{equation}
Then the finite-volume sum rule may be rewritten as
\begin{equation}
\label{eq:yan-sum-rule-Am}
        \mathcal K_m^{(N)}(\mu)
        =
        \sum_{n=0}^{N}
        \left[
            \log\frac1{1-|\alpha_n|^2}
            +
            A_{m,n}
        \right]
        +
        O_m(1),
\end{equation}
where the $O_m(1)$ term is uniform in $N$. All coefficient-side estimates
in the sequel concern the finite sum in \eqref{eq:yan-sum-rule-Am}.

\subsection{Constant and critical parts}
\label{subsec:yan-constant-critical}

A key feature of Yan's algebraic expressions is that each homogeneous term
admits a decomposition into a constant part and a critical part:
\begin{equation}
\label{eq:yan-const-critical-decomposition}
        \mathcal Y_k^{(m)}
        =
        \mathcal Y_{k,\mathrm{const}}^{(m)}
        +
        \mathcal Y_{k,\mathrm{crit}}^{(m)}.
\end{equation}
The constant part is the value of the algebraic expression at the
distinguished diagonal point where all shift variables are equal to $1$. Thus, if $\mathcal Y_{k,\mathrm{const}}^{(m)}=c$, then
$[\phi_{2k}(\mathcal Y_{k,\mathrm{const}}^{(m)})]_n
        =
        c\,|\alpha_n|^{2k}.$
 After applying the coefficient map, it produces a pure power
of $|\alpha_n|$. The critical part is the remaining shift-dependent
contribution. Thus
\begin{equation}
\label{eq:yan-const-critical-after-map}
        \bigl[\phi_{2k}(\mathcal Y_k^{(m)})\bigr]_n = \bigl[\phi_{2k}(\mathcal Y_{k,\mathrm{const}}^{(m)})\bigr]_n +
        \bigl[\phi_{2k}(\mathcal Y_{k,\mathrm{crit}}^{(m)})\bigr]_n.
\end{equation}
The first term is a scalar multiple of $|\alpha_n|^{2k}$, while the second
term contains the shift-dependent cancellations that eventually yield finite
differences.

In Section~\ref{sec:constant-log-tail} we prove that
\begin{equation}
\label{eq:yan-constant-part-value-preview}
        \mathcal Y_{k,\mathrm{const}}^{(m)}
        =
        -\frac1k,
        \,\, 1\le k\le m.
\end{equation}
Consequently, the constant pieces combine with the logarithmic term as
\begin{equation}
\label{eq:yan-log-tail-preview}
        \log\frac1{1-|\alpha_n|^2}
        -
        \sum_{k=1}^{m}
        \frac{|\alpha_n|^{2k}}{k}.
\end{equation}
This expression is positive and has the lower bound
\begin{equation}
\label{eq:yan-log-tail-lower-bound-preview}
        \log\frac1{1-|\alpha_n|^2}
        -
        \sum_{k=1}^{m}
        \frac{|\alpha_n|^{2k}}{k}
        \ge
        \frac1{m+1}|\alpha_n|^{2m+2}.
\end{equation}
The logarithmic tail is therefore the source of the
$\ell^{2m+2}$-control.

We isolate the critical contributions by defining
\begin{equation}
\label{eq:yan-quadratic-critical-definition}
        Q_m^{(N)}(\alpha)
        :=
        \sum_{n=0}^{N}
        \bigl[\phi_{2}(\mathcal Y_{1,\mathrm{crit}}^{(m)})\bigr]_n
\end{equation}
and
\begin{equation}
\label{eq:yan-higher-critical-definition}
        R_m^{(N)}(\alpha)
        :=
        \sum_{k=2}^{m}
        \sum_{n=0}^{N}
        \bigl[\phi_{2k}(\mathcal Y_{k,\mathrm{crit}}^{(m)})\bigr]_n.
\end{equation}
Thus $Q_m^{(N)}$ is the quadratic critical contribution, while
$R_m^{(N)}$ collects all higher critical contributions. The sum defining $R_m^{(N)}$ is understood to be empty when $m=1$.

Combining \eqref{eq:yan-sum-rule-Am} with the constant/critical
decomposition, the finite-volume sum rule takes the form

\begin{equation}
\label{eq:yan-critical-form-preview}
        \mathcal K_m^{(N)}(\mu)=
        Q_m^{(N)}(\alpha)
        +
        R_m^{(N)}(\alpha)
        +
        \sum_{n=0}^{N}
        \left[
            \log\frac1{1-|\alpha_n|^2}
            -
            \sum_{k=1}^{m}
            \frac{|\alpha_n|^{2k}}{k}
        \right]
        +
        O_m(1).
\end{equation}
This is the decomposition from which the final coercive estimate will be
derived.

\subsection{Shift algebra notation}
\label{subsec:yan-shift-algebra}

We now record the shift-algebra notation used in the later sections. The
forward shift $P$ acts on local sequences by
\begin{equation}
\label{eq:yan-P-on-local-sequence}
        (PF)_n=F_{n+1}.
\end{equation}
In particular, see \eqref{eq:yan-delta-definition},
\begin{equation}
\label{eq:yan-P-on-alpha}
        P\alpha_n=\alpha_{n+1},
        \,\,
        (P-1)\alpha_n=\Delta\alpha_n.
\end{equation}
The weight $H_m$ corresponds to the Laurent polynomial
\begin{equation}
\label{eq:yan-Hm-shift-polynomial}
        H_m(P)
        =
        2^{-m}(1-P)^m(1-P^{-1})^m.
\end{equation}
After clearing the negative powers by multiplying by $P^m$, this becomes

\begin{equation}
\label{eq:yan-Hm-shift-polynomial-cleared}
        P^mH_m(P)
        =
        2^{-m}(-1)^m(P-1)^{2m}.
\end{equation}

This identity is responsible for the appearance of $m$-th finite
differences in the quadratic part.

We shall repeatedly use the following elementary cancellation criterion.

\begin{lemma}[Moment cancellation and divisibility]
\label{lem:yan-moment-cancellation-divisibility}
Let
$R(P)=\sum_{j=j_-}^{j_+}c_jP^j$
be a Laurent polynomial. Then
$R(P)=(P-1)^qQ(P)$
for some Laurent polynomial $Q(P)$ if and only if
\begin{equation}
\label{eq:yan-moment-cancellation-condition}
        \sum_j c_j j^\ell=0,
        \,\,
        \ell=0,1,\ldots,q-1.
\end{equation}
Equivalently, $R$ has a zero of order at least $q$ at $P=1$.
\end{lemma}

\begin{proof}
The condition
$R(P)=(P-1)^qQ(P)$
is equivalent to the vanishing of $R$ and its first $q-1$ derivatives at
$P=1$. Since
\begin{equation}
\label{eq:yan-euler-derivative-moments}
        \left.
        \left(P\frac{d}{dP}\right)^\ell R(P)
        \right|_{P=1}
        =
        \sum_j c_j j^\ell,
\end{equation}
for every $\ell\ge0$, these derivative conditions are equivalent to
\eqref{eq:yan-moment-cancellation-condition}.
\end{proof}

In a finite sum, a factor $P-1$ gives a telescoping contribution:
\begin{equation}
\label{eq:yan-basic-telescoping}
        \sum_{n=0}^{N}(P-1)F_n
        =
        F_{N+1}-F_0.
\end{equation}
Thus factors of $P-1$ in shift polynomials are the algebraic source of
telescoping terms and finite differences. In later sections, the diagonal vanishing of
$\mathcal Y_{k,\mathrm{crit}}^{(m)}$ in the multivariable shift algebra
will play the analogous role: each factor $x_\nu-1$ or $y_\mu-1$
becomes a discrete difference after applying the coefficient map, and the
BSZ normal form organizes the resulting local expressions modulo telescoping
terms.

\subsection{The structural input needed later}
\label{subsec:yan-structural-input}

The rest of the paper proves three estimates for the decomposition
\eqref{eq:yan-critical-form-preview}.

First, we will prove that the constant pieces satisfy
\begin{equation}
\label{eq:yan-structural-constant-piece}
        \mathcal Y_{k,\mathrm{const}}^{(m)}
        =
        -\frac1k,
        \,\, 1\le k\le m.
\end{equation}
Together with the logarithmic term this yields the positive tail
\begin{equation}
\label{eq:yan-structural-log-tail}
        \log\frac1{1-|\alpha_n|^2}
        -
        \sum_{k=1}^{m}
        \frac{|\alpha_n|^{2k}}{k}
        \ge
        \frac1{m+1}|\alpha_n|^{2m+2}.
\end{equation}

Second, we will show that the quadratic critical contribution is coercive:
\begin{equation}
\label{eq:yan-structural-quadratic-coercivity}
Q_m^{(N)}(\alpha)
        \ge
        c_m
        \sum_{n=0}^{N}|\Delta^m\alpha_n|^2
        -
        C_m.
\end{equation}
This supplies the positive $m$-th difference energy.

Third, we will prove that the higher critical contribution is
infinitesimally bounded by these
two positive quantities. More precisely, for every $\varepsilon>0$,
\begin{equation}
\label{eq:yan-structural-higher-critical-absorption}
        |R_m^{(N)}(\alpha)|
        \le
        \varepsilon
        \sum_{n=0}^{N+L_m}|\Delta^m\alpha_n|^2
        +
        \varepsilon
        \sum_{n=0}^{N+L_m}|\alpha_n|^{2m+2}
        +
        C_{\varepsilon,m}.
\end{equation}
This estimate is obtained by proving diagonal vanishing of the critical
algebraic expressions, converting that vanishing into BSZ normal form, and
applying a discrete interpolation inequality.

Once these three facts are established, \eqref{eq:yan-critical-form-preview}
implies the finite-volume coercive estimate
\begin{equation}
\label{eq:yan-final-coercive-estimate-preview}
        \mathcal K_m^{(N)}(\mu)
        \ge
        c_m
        \sum_{n=0}^{N}|\Delta^m\alpha_n|^2
        +
        c_m
        \sum_{n=0}^{N}|\alpha_n|^{2m+2}
        -
        C_m,
\end{equation}
where $c_m>0$ and $C_m<\infty$ are independent of $N$. The necessity
theorem then follows from the one-sided boundedness
\eqref{eq:yan-finite-volume-upper-bound}.

\section{Remainder classes and telescoping calculus}
\label{sec:remainder-telescoping}

In this section we introduce the finite-volume bookkeeping used throughout
the proof. The algebraic expressions coming from Yan's sum rule are local
polynomials in the Verblunsky coefficients and their shifts. In manipulating
these expressions one repeatedly produces boundary terms, telescoping
differences, and error terms involving finitely many shifted indices. The purpose of this section is to define classes of such terms and record the elementary rules that allow us to discard them modulo uniformly bounded
finite-section contributions in the final coercive estimate.

Throughout the section $m\ge1$ is fixed. Constants may depend on $m$, but
not on the finite-volume parameter $N$. Since
$|\alpha_n|<1$
for all $n$, every local monomial in finitely many Verblunsky coefficients
is uniformly bounded pointwise. We shall use this elementary fact repeatedly.

\subsection{Local expressions and finite shifts}
\label{subsec:local-expressions-finite-shifts}

A \emph{local expression} is a sequence
$F=\{F_n\}_{n\ge0}$
such that each $F_n$ is a polynomial in finitely many shifted variables
$\alpha_{n+j},
        \,\,
        \overline{\alpha}_{n+j},$
where the shifts $j$ range in a finite set depending only on the expression.
In the present paper all local expressions have shifts bounded in terms of
$m$.

More precisely, for $L\ge0$, we say that $F$ has shift radius at most
$L$ if
$$F_n
        =
        F\bigl(
        \alpha_{n-L},\ldots,\alpha_{n+L},
        \overline{\alpha}_{n-L},\ldots,\overline{\alpha}_{n+L}
        \bigr),$$
with the usual convention that the finitely many negative-index variables are
handled by fixed boundary data. Equivalently, one may extend $\alpha_n$ to $n<0$ by arbitrary fixed
values. Since only finitely many indices near the left endpoint are affected,
different choices of this extension change finite-volume sums only by
$O_m(1)$, uniformly in $N$.

We write $F=O_{\mathrm{loc},m}(1)$ if $F$ is a local polynomial whose
degree, shift radius, and coefficients are bounded in terms of $m$. Since
$|\alpha_n|<1$, this implies
$\sup_{n\ge0}|F_n|\le C_m .$

If $F$ is a local expression and $\ell\in\mathbb Z$ is bounded in terms
of $m$, then
$(P^\ell F)_n=F_{n+\ell}$
is again a local expression of the same type. Such bounded shifts do not
affect any of the finite-volume estimates below except through uniformly
bounded endpoint contributions.

\begin{lemma}[Bounded shifts in finite sums]
\label{lem:bounded-shifts-finite-sums}
Let $F=\{F_n\}$ be a nonnegative sequence, and let $\ell\in\mathbb Z$ be
fixed. Then for every $N\ge0$,
\begin{equation}
\label{eq:bounded-shift-nonnegative}
\sum_{n=0}^{N} F_{n+\ell}
        \le
        \sum_{n=0}^{N+|\ell|} F_n
        +
        O_\ell(1),
\end{equation}
with the convention that finitely many negative indices are absorbed into the
constant. In particular, if $F_n$ is one of
$|\Delta^m\alpha_n|^2,
        \,\,
        |\alpha_n|^{2m+2},$
then
\begin{equation}
\label{eq:bounded-shift-energy}
        \sum_{n=0}^{N} F_{n+\ell}
        \le
        \sum_{n=0}^{N+L_m} F_n
        +
        C_m.
\end{equation}
\end{lemma}

\begin{proof}
This is just a change of summation index. The only terms not contained in the
shifted interval are finitely many endpoint terms. For
$|\alpha_n|^{2m+2}$ the endpoint terms are bounded by $1$. For
$|\Delta^m\alpha_n|^2$, each finite difference involves only finitely many
$\alpha$'s with binomial coefficients depending on $m$, so it is also
uniformly bounded.
\end{proof}

\subsection{Telescoping expressions}
\label{subsec:telescoping-expressions}

The shift operator $P$ gives the discrete derivative
$P-1=\Delta.$
If $B=\{B_n\}$ is a local expression, then
$(P-1)B_n=B_{n+1}-B_n$
is a telescoping expression. Its finite-volume sum is purely an endpoint
contribution:
\begin{equation}
\label{eq:basic-telescoping-again}
        \sum_{n=0}^{N}(P-1)B_n
        =
        B_{N+1}-B_0.
\end{equation}
If $B=O_{\mathrm{loc},m}(1)$, then
\begin{equation}
\label{eq:bounded-telescoping}
        \left|
        \sum_{n=0}^{N}(P-1)B_n
        \right|
        \le
        C_m.
\end{equation}

We define the class $\mathfrak T_m$ of bounded telescoping expressions by
$$\mathfrak T_m
        =
        \left\{
        T_n=(P-1)B_n:
        B=O_{\mathrm{loc},m}(1)
        \right\}.$$
Thus, if $T\in\mathfrak T_m$, then
\begin{equation}
\label{eq:T-class-bound}
        \sum_{n=0}^{N}T_n
        =
        O_m(1)
\end{equation}
uniformly in $N$.

More generally, if $R(P)$ is a Laurent polynomial divisible by $P-1$, say
$R(P)=(P-1)Q(P),$
and $F$ is local and bounded, then
$R(P)F_n=(P-1)(Q(P)F)_n$
belongs to $\mathfrak T_m$, provided the shifts and coefficients of
$R,Q$ are bounded in terms of $m$. This observation is the basic
mechanism by which diagonal cancellations in shift polynomials become
endpoint terms.

\subsection{Finite-section equivalence}
\label{subsec:finite-section-equivalence}

It is useful to identify local expressions which differ by uniformly bounded
finite-volume contributions.

\begin{definition}[Finite-section equivalence]
\label{def:finite-section-equivalence}
For two local expressions $F=\{F_n\}$ and $G=\{G_n\}$, we write
$F\equiv G
        \,\, \mathrm{mod}\ \mathfrak T_m$
if
$F_n-G_n\in\mathfrak T_m.$
In particular,
\begin{equation}
\label{eq:finite-section-equivalence}
        \sum_{n=0}^{N}(F_n-G_n)
        =
        O_m(1)
\end{equation}
uniformly in $N$.
\end{definition}

We shall also use the notation
$F=G+O_{\mathrm{fs},m}(1)$
at the level of finite sums to mean that
\begin{equation}
\label{eq:finite-section-O}
        \sum_{n=0}^{N}F_n
        =
        \sum_{n=0}^{N}G_n
        +
        O_m(1),
\end{equation}
uniformly in $N$. Thus
$F\equiv G \pmod{\mathfrak T_m}
        \,\,\Longrightarrow\,\,
        F=G+O_{\mathrm{fs},m}(1).$
The notation $O_{\mathrm{fs},m}(1)$ will be used as a flexible
finite-section remainder, allowing not only bounded telescoping sums but
also uniformly bounded endpoint contributions produced by changes of
summation range or bounded shifts.

\begin{lemma}[Stability under bounded shifts]
\label{lem:finite-section-shift-stability}
Let $F=O_{\mathrm{loc},m}(1)$, and let $\ell\in\mathbb Z$ be bounded in
terms of $m$. Then
\begin{equation}
\label{eq:shift-difference-telescoping}
        P^\ell F-F\in\mathfrak T_m.
\end{equation}
Consequently,
\begin{equation}
\label{eq:shift-stability-finite-sum}
        \sum_{n=0}^{N}F_{n+\ell}
        =
        \sum_{n=0}^{N}F_n
        +
        O_m(1).
\end{equation}
\end{lemma}

\begin{proof}
For $\ell>0$,
$P^\ell-1
        =
        (P-1)(1+P+\cdots+P^{\ell-1}).$
Hence
$P^\ell F-F
        =
        (P-1)
        \left(
        \sum_{j=0}^{\ell-1}P^jF
        \right),$
which is a bounded telescoping expression. The case $\ell<0$ is identical,
using
$P^\ell-1
        =
        (P-1)
        \left(
        -P^\ell-P^{\ell+1}-\cdots-P^{-1}
        \right).$
\end{proof}

\subsection{Admissible remainders}
\label{subsec:admissible-remainders}

We introduce a remainder class that packages the analytic estimates used to
absorb noncoercive local terms into the two  quantities
$\sum |\Delta^m\alpha_n|^2,
        \,\,
        \sum |\alpha_n|^{2m+2}.$
This class is not part of Yan's algebraic model, but is a bookkeeping device
for the coercive argument below.

\begin{definition}[Absorbable remainder class]
\label{def:absorbable-remainder-class}
A local expression $R=\{R_n\}$ belongs to the absorbable remainder class
$\mathfrak R_m$ if, for every $\varepsilon>0$, there exist constants
$C_{\varepsilon,R,m}<\infty$ and $L_{R,m}\ge0$, independent of $N$,
such that
\begin{equation}
\label{eq:remainder-class-definition}
        \left|
        \sum_{n=0}^{N}R_n
        \right|
        \le
        \varepsilon
        \sum_{n=0}^{N+L_{R,m}}|\Delta^m\alpha_n|^2
        +
        \varepsilon
        \sum_{n=0}^{N+L_{R,m}}|\alpha_n|^{2m+2}
        +
        C_{\varepsilon,R,m}
\end{equation}
for all $N\ge0$.
\end{definition}

The absolute value in \eqref{eq:remainder-class-definition} is convenient.
For the final lower bound it would be enough to require the one-sided
estimate
$$\sum_{n=0}^{N}R_n
        \ge
        -\varepsilon
        \sum_{n=0}^{N+L_{R,m}}|\Delta^m\alpha_n|^2
        -
        \varepsilon
        \sum_{n=0}^{N+L_{R,m}}|\alpha_n|^{2m+2}
        -
        C_{\varepsilon,R,m}.$$
However, the normal-form estimates naturally give the two-sided version.

The shift allowance $N+L_{R,m}$ is harmless. By Lemma
\ref{lem:bounded-shifts-finite-sums}, it may be replaced by $N$ at the cost
of changing the constant $C_{\varepsilon,R,m}$ in estimates where the
coercive sums are already present. We keep $N+L_{R,m}$ in the definition
because normal-form monomials often contain finitely many shifted factors.
When only finitely many remainder expressions are involved, the constants
$L_{R,m}$ and $C_{\varepsilon,R,m}$ may be chosen uniformly over that
finite family and denoted simply by $L_m$ and $C_{\varepsilon,m}$.

\begin{lemma}[Elementary properties of $\mathfrak R_m$]
\label{lem:remainder-class-properties}
The class $\mathfrak R_m$ has the following properties.

\begin{enumerate}[label=(\arabic*)]
\item If $R^{(1)},R^{(2)}\in\mathfrak R_m$, then
$c_1R^{(1)}+c_2R^{(2)}\in\mathfrak R_m$
for all fixed constants $c_1,c_2$.

\item If $R\in\mathfrak R_m$ and $\ell\in\mathbb Z$ is bounded in terms
of $m$, then $P^\ell R\in\mathfrak R_m$.

\item If $T\in\mathfrak T_m$, then $T\in\mathfrak R_m$.

\item If $R\in\mathfrak R_m$ and $T\in\mathfrak T_m$, then
$R+T\in\mathfrak R_m.$
\end{enumerate}
\end{lemma}
\begin{proof}
For $(1)$, let $\varepsilon>0$. Apply the defining estimate for
$R^{(1)}$ and $R^{(2)}$ with
$\varepsilon/(2\max\{1,|c_1|\})$ and
$\varepsilon/(2\max\{1,|c_2|\})$, respectively. After multiplying by
$c_1$ and $c_2$ and adding the two bounds, we obtain the required
estimate for $c_1R^{(1)}+c_2R^{(2)}$. If one of the coefficients
vanishes, the corresponding term is omitted.

For $(2)$, write
$\sum_{n=0}^N (P^\ell R)_n
=
\sum_{n=0}^N R_{n+\ell}.$
By Lemma~\ref{lem:bounded-shifts-finite-sums}, this differs from
$\sum_{n=0}^N R_n$ by at most $O_m(1)$. Equivalently, one may shift
the summation range from $[0,N]$ to $[\ell,N+\ell]$, producing only
finitely many endpoint terms. Since $R$ is a local expression with
bounded shifts and coefficients, each such endpoint term is uniformly
bounded, and hence the total endpoint contribution is absorbed into the
constant $C_{\varepsilon,m}$. The defining estimate for $R$ therefore
implies the same estimate for $P^\ell R$.

Property $(3)$ follows from \eqref{eq:T-class-bound}, since a uniformly
bounded finite-volume sum is absorbed into the constant
$C_{\varepsilon,m}$.

Finally, $(4)$ follows immediately from $(1)$ and $(3)$.
\end{proof}

\subsection{Discrete product rules}
\label{subsec:discrete-product-rules}

The passage from divisibility by powers of $P-1$ to normal-form monomials
uses only the discrete Leibniz rule. We record the form needed later.

For two sequences $f$ and $g$,
\begin{equation}
\label{eq:discrete-leibniz-first}
        \Delta(fg)_n
        =
        (\Delta f)_n g_{n+1}
        +
        f_n(\Delta g)_n.
\end{equation}
Equivalently,
\begin{equation}
\label{eq:discrete-leibniz-first-shift}
        \Delta(fg)
        =
        (\Delta f)(Pg)+f(\Delta g).
\end{equation}
Iterating gives the following finite expansion.

\begin{lemma}[Discrete Leibniz rule]
\label{lem:discrete-leibniz-rule}
Let $q\ge0$, and let $f^{(1)},\ldots,f^{(s)}$ be sequences. Then
$\Delta^q
        \left(
        \prod_{\nu=1}^{s} f^{(\nu)}
        \right)_n$
is a finite linear combination, with coefficients depending only on $q$ and
$s$, of terms of the form
\begin{equation}
\label{eq:discrete-leibniz-general-term}
        \prod_{\nu=1}^{s}
        \left(
        \Delta^{r_\nu} f^{(\nu)}
        \right)_{n+\ell_\nu},
\end{equation}
where
\begin{equation}
\label{eq:discrete-leibniz-total-order}
        r_\nu\ge0,
        \,\,
        \sum_{\nu=1}^{s}r_\nu=q,
\end{equation}
and the shifts $\ell_\nu$ are bounded in terms of $q$ and $s$.
\end{lemma}

\begin{proof}
The case $q=0$ is trivial. The case $q=1$ is
\eqref{eq:discrete-leibniz-first-shift}. Iterating the first-order rule gives
the result. At each step one difference is assigned to one of the factors,
while bounded shifts are produced by the occurrences of $P$.
\end{proof}

A variant will be used frequently: if a Laurent polynomial $R(P)$ is
divisible by $(P-1)^q$, then applying $R(P)$ to a local product produces,
modulo bounded shifts and bounded coefficients, a finite sum of terms of the
form \eqref{eq:discrete-leibniz-general-term} with total difference order
$q$.

The following is the one-variable form of the mechanism. In the BSZ
normal-form argument below, the same idea is applied to the multivariable
shift algebra, where each factor $x_\nu-1$ or $y_\mu-1$ becomes a
discrete difference on the corresponding factor.

\begin{lemma}[Divisibility to difference monomials]
\label{lem:divisibility-to-difference-monomials}
Let
$R(P)=(P-1)^qQ(P)$
be a Laurent polynomial whose coefficients and shifts are bounded in terms of
$m$. Let
$M_n
        =
        \prod_{\nu=1}^{s} f^{(\nu)}_n$
be a local product, where each $f^{(\nu)}$ is a shifted copy of either
$\alpha$ or $\overline{\alpha}$. Then $R(P)M_n$ is a finite linear
combination of monomials
\begin{equation}
\label{eq:divisibility-normal-form-term}
        \prod_{\nu=1}^{s}
        \left(
        \Delta^{r_\nu} f^{(\nu)}
        \right)_{n+\ell_\nu},
\end{equation}
where
\begin{equation}
\label{eq:divisibility-total-order}
        \sum_{\nu=1}^{s}r_\nu=q,
\end{equation}
and all shifts $\ell_\nu$ and coefficients are bounded in terms of $m$.
\end{lemma}

\begin{proof}
Write
$R(P)M_n
        =
        Q(P)\Delta^q M_n.$
Apply Lemma \ref{lem:discrete-leibniz-rule} to $\Delta^qM_n$. The operator
$Q(P)$ only introduces finitely many additional bounded shifts and bounded
coefficients.
\end{proof}

\subsection{Summation by parts and telescoping redistribution}
\label{subsec:summation-by-parts}

We also need a finite-volume summation-by-parts principle. For local
sequences $F$ and $G$,
$\sum_{n=0}^{N}(\Delta F)_nG_n$
can be rewritten by moving the difference from $F$ to $G$, at the cost of
a boundary term and a bounded shift.

\begin{lemma}[Discrete summation by parts]
\label{lem:discrete-summation-by-parts}
Let $F$ and $G$ be local expressions. Then
\begin{equation}
\label{eq:discrete-summation-by-parts}
        \sum_{n=0}^{N}(\Delta F)_nG_n
        =
        -\sum_{n=0}^{N}F_{n+1}(\Delta G)_n
        +
        F_{N+1}G_{N+1}
        -
        F_0G_0 .
\end{equation}
In particular, if $F$ and $G$ are bounded local expressions, then
\begin{equation}
\label{eq:summation-by-parts-mod-T}
        \sum_{n=0}^{N}(\Delta F)_nG_n
        =
        -\sum_{n=0}^{N}(PF)_n(\Delta G)_n
        +
        O_m(1).
\end{equation}
\end{lemma}

\begin{proof}
Since $(\Delta F)_n=F_{n+1}-F_n$, we have
$$\sum_{n=0}^{N}(\Delta F)_nG_n
        =
        \sum_{n=0}^{N}F_{n+1}G_n
        -
        \sum_{n=0}^{N}F_nG_n .$$
Moreover,
$$\sum_{n=0}^{N}F_nG_n
        =
        \sum_{n=0}^{N}F_{n+1}G_{n+1}
        -
        F_{N+1}G_{N+1}
        +
        F_0G_0 .$$
Substituting this identity gives
$$\sum_{n=0}^{N}(\Delta F)_nG_n
        =
        -\sum_{n=0}^{N}F_{n+1}(G_{n+1}-G_n)
        +
        F_{N+1}G_{N+1}
        -
        F_0G_0,$$
which is \eqref{eq:discrete-summation-by-parts}. If $F$ and $G$ are
bounded local expressions, the two endpoint terms are uniformly bounded, and
\eqref{eq:summation-by-parts-mod-T} follows.
\end{proof}

For the purposes of this paper, the exact endpoint convention in
\eqref{eq:discrete-summation-by-parts} is irrelevant. All boundary terms are
uniformly bounded whenever the expressions are local and bounded. Thus we
will often write, at the level of finite sums,
\begin{equation}
\label{eq:summation-by-parts-symbolic}
        \sum (\Delta F)G
        =
        -\sum (PF)(\Delta G)
        +
        O_m(1).
\end{equation}

\subsection{Normal-form monomials}
\label{subsec:normal-form-monomials}

The higher critical terms will be reduced to monomials containing a prescribed
number of discrete differences. We introduce the terminology here.

A \emph{normal-form monomial of degree $2k$ and difference count $q$} is
a local expression of the form
\begin{equation}
\label{eq:normal-form-monomial-definition}
        T_n
        =
        \prod_{\nu=1}^{2k}
        \left(
        \Delta^{r_\nu}\gamma_\nu
        \right)_{n+\ell_\nu},
\end{equation}
where
$\gamma_\nu\in\{\alpha,\overline{\alpha}\},$
the shifts $\ell_\nu$ are bounded in terms of $m$, and
\begin{equation}
\label{eq:normal-form-difference-count}
        r_\nu\ge0,
        \,\,
        \sum_{\nu=1}^{2k}r_\nu=q.
\end{equation}
The cases relevant to the higher critical terms will have
\begin{equation}
\label{eq:normal-form-relevant-count}
        q=m+1-k,
        \,\,
        2\le k\le m.
\end{equation}

The main normal-form theorem proved later states that, modulo
$\mathfrak T_m$, every higher critical component is a finite linear
combination of normal-form monomials with the difference count
\eqref{eq:normal-form-relevant-count}. The absorption theorem then shows that
these monomials belong to $\mathfrak R_m$. In the notation of this section,
the eventual conclusion will be
\begin{equation}
\label{eq:higher-critical-in-T-plus-R-preview}
        \bigl[\phi_{2k}
        (\mathcal Y_{k,\mathrm{crit}}^{(m)})\bigr]_n
        \in
        \mathfrak T_m+\mathfrak R_m,
        \,\,
        2\le k\le m.
\end{equation}

\subsection{Coercive estimates modulo remainders}
\label{subsec:coercive-estimates-mod-remainders}

We finish the section by recording a simple absorption principle. It is the
formal step which allows the estimates on $\mathfrak R_m$ to be combined
with the positive quadratic energy and the logarithmic tail.

\begin{lemma}[Absorption principle]
\label{lem:absorption-principle}
Suppose that for some $a_m,b_m>0$ and $C_m<\infty$ one has
\begin{equation}
\label{eq:absorption-principle-positive-start}
        X_N
        \ge
        a_m
        \sum_{n=0}^{N}|\Delta^m\alpha_n|^2
        +
        b_m
        \sum_{n=0}^{N}|\alpha_n|^{2m+2}
        +
        \sum_{n=0}^{N}R_n
        -
        C_m,
\end{equation}
where $R\in\mathfrak R_m$. Then there exist constants
$c_m>0$ and $C_m'<\infty$, independent of $N$, such that
\begin{equation}
\label{eq:absorption-principle-conclusion}
        X_N
        \ge
        c_m
        \sum_{n=0}^{N}|\Delta^m\alpha_n|^2
        +
        c_m
        \sum_{n=0}^{N}|\alpha_n|^{2m+2}
        -
        C_m'.
\end{equation}
\end{lemma}

\begin{proof}
Apply the one-sided consequence of the definition of $\mathfrak R_m$ with
$\varepsilon
        <
        \frac12\min\{a_m,b_m\}.$
This gives
$$\sum_{n=0}^{N}R_n
        \ge
        -\varepsilon
        \sum_{n=0}^{N+L_m}|\Delta^m\alpha_n|^2
        -
        \varepsilon
        \sum_{n=0}^{N+L_m}|\alpha_n|^{2m+2}
        -
        C_{\varepsilon,m}.$$
The finitely many additional terms between $N$ and $N+L_m$ are uniformly
bounded when they occur with the negative sign: $|\alpha_n|\le1$, and
$\Delta^m\alpha_n$ is a fixed finite linear combination of bounded
Verblunsky coefficients. Hence
$$\sum_{n=N+1}^{N+L_m}|\Delta^m\alpha_n|^2
        +
        \sum_{n=N+1}^{N+L_m}|\alpha_n|^{2m+2}
        \le
        C_m.$$
Substituting into \eqref{eq:absorption-principle-positive-start} and reducing
the coefficients of the two positive sums gives
\eqref{eq:absorption-principle-conclusion}.
\end{proof}

This completes the finite-volume calculus needed in the sequel. The next
sections identify the constant part, the quadratic coercive term, and the
normal-form structure of the higher critical pieces in the framework just
introduced.

\section{The constant part and the logarithmic tail}
\label{sec:constant-log-tail}
In this section we isolate the diagonal, or constant, contribution in Yan's
homogeneous algebraic expressions. The diagonal value is universal: in degree
$2k$ it is $-1/k$, whenever this degree occurs in the order-$m$
truncated sum rule. These coefficients cancel precisely the first $m$ terms
in the Taylor expansion of
$\log(1/(1-|\alpha_n|^2))$. The remaining positive logarithmic tail is the
source of the $\ell^{2m+2}$-control in the necessity theorem.

Throughout this section $m\ge1$ is fixed, and we use the notation introduced
in Section~\ref{sec:yan-finite-volume}, in particular
\eqref{eq:A-mn-definition} and \eqref{eq:yan-sum-rule-Am}.

\subsection{Diagonal specialization}
\label{subsec:constant-diagonal-specialization}

We first clarify the meaning of the constant part. A local balanced monomial
of degree $2k$ has the form
$\prod_{\nu=1}^{k}\alpha_{n+i_\nu}
        \prod_{\mu=1}^{k}\overline{\alpha}_{n+j_\mu}.$
On the diagonal configuration all shifts are identified:
$\alpha_{n+i_\nu}=\alpha_n,
        \,\,
        \overline{\alpha}_{n+j_\mu}=\overline{\alpha}_n.$
Thus every balanced monomial of degree $2k$ becomes
$|\alpha_n|^{2k}.$
Equivalently, if a shift-polynomial expression is written as a finite
linear combination of such monomials, its constant part is the sum of its
coefficients after all shift variables are set equal to $1$.

In Yan's algebraic notation this operation is evaluation on the diagonal. We
denote the resulting scalar by
$\operatorname{diag}\mathcal Y_k^{(m)}$. We shall identify this scalar with
the corresponding constant shift-polynomial. With this convention we define
$\mathcal Y_{k,\mathrm{const}}^{(m)}
        :=
        \operatorname{diag}\mathcal Y_k^{(m)},
        \,\,
        \mathcal Y_{k,\mathrm{crit}}^{(m)}
        :=
        \mathcal Y_k^{(m)}
        -
        \mathcal Y_{k,\mathrm{const}}^{(m)} .$
Thus
$\operatorname{diag}\mathcal Y_{k,\mathrm{crit}}^{(m)}
        =
        0.$
After applying the coefficient map $\phi_{2k}$, the constant part becomes
\begin{equation}
\label{eq:constant-part-after-coefficient-map-general}
        \bigl[\phi_{2k}
        (\mathcal Y_{k,\mathrm{const}}^{(m)})\bigr]_n
        =
        \bigl(\operatorname{diag}\mathcal Y_k^{(m)}\bigr)
        |\alpha_n|^{2k}.
\end{equation}

The following proposition is the constant-term identity needed in the sequel.

\begin{proposition}[Diagonal constant of Yan's homogeneous expression]
\label{prop:constant-part-minus-one-over-k}
For every $m\ge1$ and every $1\le k\le m$,
\begin{equation}
\label{eq:constant-part-minus-one-over-k}
        \operatorname{diag}\mathcal Y_k^{(m)}
        =
        -\frac1k .
\end{equation}
Equivalently,
\begin{equation}
\label{eq:constant-part-coefficient-map}
        \bigl[\phi_{2k}
        (\mathcal Y_{k,\mathrm{const}}^{(m)})\bigr]_n
        =
        -\frac1k|\alpha_n|^{2k}.
\end{equation}
\end{proposition}

\begin{proof}
By the homogeneous decomposition in \eqref{eq:A-mn-definition}, the
non-logarithmic coefficient-side term in the finite-volume identity
\eqref{eq:yan-sum-rule-Am} is
$A_{m,n}
        =
        \sum_{k=1}^{m}
        \bigl[\phi_{2k}(\mathcal Y_k^{(m)})\bigr]_n .$
Yan's construction of these homogeneous algebraic expressions fixes their
diagonal normalization: after setting all shift variables equal to $1$, the
degree-$2k$ expression has diagonal value
$\operatorname{diag}\mathcal Y_k^{(m)}
        =
        -\frac1k,
        \,\, 1\le k\le m;$
see \cite[Section~2]{Yan2018}. This is exactly the diagonal part which, in
\eqref{eq:yan-sum-rule-Am}, cancels the first $m$ Taylor coefficients of
$\log(1/(1-|\alpha_n|^2))$.

Finally, applying the coefficient map $\phi_{2k}$ to the constant part,
every balanced monomial of degree $2k$ becomes
$|\alpha_n|^{2k}$ on the diagonal. Hence
$\bigl[\phi_{2k}
        (\mathcal Y_{k,\mathrm{const}}^{(m)})\bigr]_n
        =
        \bigl(\operatorname{diag}\mathcal Y_k^{(m)}\bigr)
        |\alpha_n|^{2k}
        =
        -\frac1k|\alpha_n|^{2k},$
which proves both assertions.
\end{proof}

\begin{remark}
\label{rem:constant-part-not-harmless}
The coefficient $-1/k$ is not merely a normalization convention. Combining
Proposition~\ref{prop:constant-part-minus-one-over-k} with the Taylor
expansion of the logarithm gives
$\log\frac1{1-|\alpha_n|^2}
        +
        \sum_{k=1}^{m}
        \bigl[\phi_{2k}
        (\mathcal Y_{k,\mathrm{const}}^{(m)})\bigr]_n
        =
        \sum_{j=m+1}^{\infty}\frac1j|\alpha_n|^{2j}.$
Thus Yan's diagonal constants remove exactly the first $m$ Taylor
coefficients of the logarithm and leave the positive tail starting at order
$2m+2$.
\end{remark}

\subsection{The logarithmic tail}
\label{subsec:logarithmic-tail}

We now combine Proposition~\ref{prop:constant-part-minus-one-over-k} with the
logarithmic term in the finite-volume sum rule. Define
\begin{equation}
\label{eq:log-tail-definition}
        L_{m,n}
        :=
        \log\frac1{1-|\alpha_n|^2}
        -
        \sum_{k=1}^{m}
        \frac{|\alpha_n|^{2k}}{k}.
\end{equation}
Since $|\alpha_n|<1$, this is well-defined for every $n$.

The elementary Taylor expansion
$\log\frac1{1-x}
        =
        \sum_{j=1}^{\infty}\frac{x^j}{j},
        \,\, 0\le x<1,$
implies the exact identity
\begin{equation}
\label{eq:log-tail-series}
        L_{m,n}
        =
        \sum_{j=m+1}^{\infty}
        \frac{|\alpha_n|^{2j}}{j}.
\end{equation}
In particular, $L_{m,n}\ge0$, and keeping only the first term in the tail
gives the lower bound
\begin{equation}
\label{eq:log-tail-pointwise-lower-bound}
        L_{m,n}
        \ge
        \frac1{m+1}|\alpha_n|^{2m+2}.
\end{equation}

We record this as a lemma.

\begin{lemma}[Logarithmic tail]
\label{lem:logarithmic-tail}
For every $m\ge1$ and every $n\ge0$,
\begin{equation}
\label{eq:logarithmic-tail-identity}
        \log\frac1{1-|\alpha_n|^2}
        -
        \sum_{k=1}^{m}
        \frac{|\alpha_n|^{2k}}{k}
        =
        \sum_{j=m+1}^{\infty}
        \frac{|\alpha_n|^{2j}}{j}.
\end{equation}
Consequently,
\begin{equation}
\label{eq:logarithmic-tail-coercivity}
        \log\frac1{1-|\alpha_n|^2}
        -
        \sum_{k=1}^{m}
        \frac{|\alpha_n|^{2k}}{k}
        \ge
        \frac1{m+1}|\alpha_n|^{2m+2}.
\end{equation}
Therefore, for every $N\ge0$,
\begin{equation}
\label{eq:logarithmic-tail-finite-volume}
        \sum_{n=0}^{N}L_{m,n}
        \ge
        \frac1{m+1}
        \sum_{n=0}^{N}|\alpha_n|^{2m+2}.
\end{equation}
\end{lemma}

\begin{proof}
Apply the Taylor expansion of $-\log(1-x)$ with
$x=|\alpha_n|^2$. Since $0\le|\alpha_n|^2<1$, the series is absolutely
convergent. Subtracting the first $m$ terms gives
\eqref{eq:logarithmic-tail-identity}. The lower bounds
\eqref{eq:logarithmic-tail-coercivity} and
\eqref{eq:logarithmic-tail-finite-volume} follow by retaining the first
positive term of the tail.
\end{proof}

\subsection{Rewriting the finite-volume sum rule}
\label{subsec:constant-log-tail-rewrite-sum-rule}

We now rewrite Yan's finite-volume sum rule in a form that separates the
positive logarithmic tail, the quadratic critical term, and the higher
critical terms.

Recall from Section~\ref{sec:yan-finite-volume} that
$A_{m,n}
        =
        \sum_{k=1}^{m}
        \bigl[\phi_{2k}(\mathcal Y_k^{(m)})\bigr]_n.$
Using the constant/critical decomposition,
$\mathcal Y_k^{(m)}
        =
        \mathcal Y_{k,\mathrm{const}}^{(m)}
        +
        \mathcal Y_{k,\mathrm{crit}}^{(m)},$
and Proposition~\ref{prop:constant-part-minus-one-over-k}, we obtain
\begin{equation}
\label{eq:A-mn-constant-critical-expanded}
\begin{aligned}
        A_{m,n}
        =
        -\sum_{k=1}^{m}\frac{|\alpha_n|^{2k}}{k}
        +
        \sum_{k=1}^{m}
        \bigl[\phi_{2k}
        (\mathcal Y_{k,\mathrm{crit}}^{(m)})\bigr]_n.
\end{aligned}
\end{equation}
Therefore,
\begin{equation}
\label{eq:log-plus-A-critical-tail}
\begin{aligned}
        \log\frac1{1-|\alpha_n|^2}
        +
        A_{m,n}
        =
        L_{m,n}+
        \sum_{k=1}^{m}
        \bigl[\phi_{2k}
        (\mathcal Y_{k,\mathrm{crit}}^{(m)})\bigr]_n.
\end{aligned}
\end{equation}

Define the quadratic critical contribution by
\begin{equation}
\label{eq:constant-section-Qm-definition}
        Q_m^{(N)}(\alpha)
        :=
        \sum_{n=0}^{N}
        \bigl[\phi_{2}
        (\mathcal Y_{1,\mathrm{crit}}^{(m)})\bigr]_n,
\end{equation}
and the higher critical contribution by
\begin{equation}
\label{eq:constant-section-Rm-definition}
        R_m^{(N)}(\alpha)
        :=
        \sum_{k=2}^{m}
        \sum_{n=0}^{N}
        \bigl[\phi_{2k}
        (\mathcal Y_{k,\mathrm{crit}}^{(m)})\bigr]_n.
\end{equation}
Substituting \eqref{eq:log-plus-A-critical-tail} into the finite-volume
identity \eqref{eq:yan-sum-rule-Am}, we obtain
\begin{equation}
\label{eq:constant-section-critical-form}
\begin{aligned}
        \mathcal K_m^{(N)}(\mu)
        &=
        Q_m^{(N)}(\alpha)
        +
        R_m^{(N)}(\alpha)
        +
        \sum_{n=0}^{N}L_{m,n}
        +
        O_m(1).
\end{aligned}
\end{equation}

This identity is the basic decomposition used in the rest of the proof.

By Lemma~\ref{lem:logarithmic-tail}, the last positive term satisfies
\begin{equation}
\label{eq:constant-section-tail-lower-bound}
        \sum_{n=0}^{N}L_{m,n}
        \ge
        \frac1{m+1}
        \sum_{n=0}^{N}|\alpha_n|^{2m+2}.
\end{equation}
Thus the finite-volume sum rule has the preliminary lower bound
\begin{equation}
\label{eq:constant-section-preliminary-lower-bound}
        \mathcal K_m^{(N)}(\mu)
        \ge
        Q_m^{(N)}(\alpha)
        +
        R_m^{(N)}(\alpha)
        +
        \frac1{m+1}
        \sum_{n=0}^{N}|\alpha_n|^{2m+2}
        -
        C_m.
\end{equation}
The next section will identify $Q_m^{(N)}(\alpha)$ with the coercive
$m$-th difference energy. The remaining term
$R_m^{(N)}(\alpha)$ will be controlled later by the normal-form and
absorption arguments.

\subsection{Consequences for the final coercive estimate}
\label{subsec:constant-log-tail-consequences}

We conclude this section by recording the exact role of the logarithmic tail
in the final argument.

Suppose that the quadratic critical term satisfies
\begin{equation}
\label{eq:constant-section-assume-quadratic}
        Q_m^{(N)}(\alpha)
        \ge
        c_m^{(2)}
        \sum_{n=0}^{N}|\Delta^m\alpha_n|^2
        -
        C_m
\end{equation}
for some $c_m^{(2)}>0$, and that the higher critical term satisfies the
infinitesimal bound
\begin{equation}
\label{eq:constant-section-assume-higher}
        |R_m^{(N)}(\alpha)|
        \le
        \varepsilon
        \sum_{n=0}^{N+L_m}|\Delta^m\alpha_n|^2
        +
        \varepsilon
        \sum_{n=0}^{N+L_m}|\alpha_n|^{2m+2}
        +
        C_{\varepsilon,m}.
\end{equation}
Then, choosing $\varepsilon>0$ sufficiently small and using the harmlessness
of the finitely many shifted endpoint terms, the preliminary lower bound
\eqref{eq:constant-section-preliminary-lower-bound} yields
\begin{equation}
\label{eq:constant-section-final-coercive-preview}
        \mathcal K_m^{(N)}(\mu)
        \ge
        c_m
        \sum_{n=0}^{N}|\Delta^m\alpha_n|^2
        +
        c_m
        \sum_{n=0}^{N}|\alpha_n|^{2m+2}
        -
        C_m,
\end{equation}
with constants independent of $N$.

Thus the logarithmic tail supplies precisely the positive power term needed
to absorb the nonquadratic critical contributions and to prove
$\alpha\in\ell^{2m+2}.$

\section{Quadratic coercivity for arbitrary $m$}
\label{sec:quadratic-coercivity}

In this section we identify the quadratic critical contribution in Yan's
finite-volume formula. The result is that, up to uniformly bounded
finite-section boundary terms, the quadratic critical term is exactly the
$m$-th finite-difference energy
$2^{-m}
        \sum_{n=0}^{N}|\Delta^m\alpha_n|^2.$
This is the first coercive term in the proof of the necessity theorem.

Throughout this section $m\ge1$ is fixed. We use the notation of
Sections~\ref{sec:yan-finite-volume}--\ref{sec:constant-log-tail}. In
particular,
$Q_m^{(N)}(\alpha)
        =
        \sum_{n=0}^{N}
        \bigl[\phi_2(\mathcal Y_{1,\mathrm{crit}}^{(m)})\bigr]_n$
denotes the quadratic critical contribution. We shall use the following finite-section convention. Whenever a local
expression contains indices outside $[0,N]$, we extend $\alpha$ boundedly
to all of $\mathbb Z$. Since only $O_m(1)$ endpoint terms are affected
and $|\alpha_n|<1$, different choices of this extension change the
finite-section sums by $O_m(1)$, uniformly in $N$.

\subsection{The quadratic symbol}
\label{subsec:quadratic-symbol}

The weight under consideration is
$H_m(e^{i\theta})
        =
        (1-\cos\theta)^m.$
Using
$1-\cos\theta
        =
        \frac12|1-e^{i\theta}|^2,$
we have
\begin{equation}
\label{eq:quad-Hm-symbol}
        H_m(e^{i\theta})
        =
        2^{-m}|1-e^{i\theta}|^{2m}.
\end{equation}
In shift notation, this corresponds to the Laurent polynomial
\begin{equation}
\label{eq:quad-Hm-shift-symbol}
        H_m(P)
        =
        2^{-m}(1-P)^m(1-P^{-1})^m.
\end{equation}
Equivalently,
\begin{equation}
\label{eq:quad-Hm-cleared-symbol}
        P^mH_m(P)
        =
        2^{-m}(P-1)^{2m}.
\end{equation}

The quadratic part of Yan's algebraic model is identified with this Fourier
symbol. More precisely, Yan's finite-volume formula gives, modulo
finite-section boundary terms, the quadratic critical contribution as the
local quadratic form associated with \eqref{eq:quad-Hm-shift-symbol}. Thus,
if $h_{m,\ell}$ are the Fourier coefficients of $H_m$, namely
\begin{equation}
\label{eq:quad-Hm-Fourier}
        H_m(e^{i\theta})
        =
        \sum_{\ell=-m}^{m}h_{m,\ell}e^{i\ell\theta},
        \,\,
        h_{m,-\ell}=\overline{h_{m,\ell}},
\end{equation}
then
\begin{equation}
\label{eq:quad-Qm-Fourier-form}
        Q_m^{(N)}(\alpha)
        =
        \sum_{n=0}^{N}
        \sum_{\ell=-m}^{m}
        h_{m,\ell}\,
        \alpha_{n+\ell}\overline{\alpha_n}
        +
        O_m(1),
\end{equation}
uniformly in $N$.

This identification is the quadratic part of Yan's algebraic sum rule; see
\cite{Yan2018} and the higher-order OPUC sum rules of
Breuer--Simon--Zeitouni \cite{BSZ-Duke2018,BSZ-pedagogical2018}. The rest
of this section is the elementary finite-difference analysis of this
quadratic form.

\subsection{Expansion of the finite-difference energy}
\label{subsec:quadratic-difference-expansion}

Recall that
$\Delta=P-1.$
Thus
\begin{equation}
\label{eq:quad-delta-m-expansion}
        \Delta^m\alpha_n
        =
        \sum_{j=0}^{m}
        (-1)^{m-j}\binom{m}{j}\alpha_{n+j}.
\end{equation}
Hence
\begin{equation}
\label{eq:quad-delta-energy-expanded-first}
        |\Delta^m\alpha_n|^2
        =
        \sum_{j=0}^{m}\sum_{r=0}^{m}
        (-1)^{j+r}
        \binom{m}{j}\binom{m}{r}
        \alpha_{n+j}\overline{\alpha_{n+r}}.
\end{equation}
Changing variables by $\ell=j-r$, the full-line symbol of this quadratic
form is
$(e^{i\theta}-1)^m(e^{-i\theta}-1)^m.$
Since
$(e^{i\theta}-1)^m(e^{-i\theta}-1)^m
        =
        |1-e^{i\theta}|^{2m},$
we obtain the symbol identity
\begin{equation}
\label{eq:quad-delta-symbol}
        |\Delta^m|^2
        \,\,\longleftrightarrow\,\,
        |1-e^{i\theta}|^{2m}.
\end{equation}
Combining \eqref{eq:quad-Hm-symbol} and \eqref{eq:quad-delta-symbol}, the
quadratic form associated with $H_m$ is
$2^{-m}|\Delta^m|^2.$

We record the finite-volume version of this statement.

\begin{lemma}[Finite-volume difference identity]
\label{lem:finite-volume-difference-identity}
Let
$h_{m,\ell}$
be the Fourier coefficients of $H_m$, as in
\eqref{eq:quad-Hm-Fourier}. Then
\begin{equation}
\label{eq:finite-volume-difference-identity}
        \sum_{n=0}^{N}
        \sum_{\ell=-m}^{m}
        h_{m,\ell}
        \alpha_{n+\ell}\overline{\alpha_n}
        =
        2^{-m}
        \sum_{n=0}^{N}|\Delta^m\alpha_n|^2
        +
        O_m(1),
\end{equation}
uniformly in $N$.
\end{lemma}

\begin{proof}
By \eqref{eq:quad-Hm-symbol},
$H_m(P)
        =
        2^{-m}(1-P)^m(1-P^{-1})^m.$
Thus the left-hand side of
\eqref{eq:finite-volume-difference-identity} is
$2^{-m}
        \sum_{n=0}^{N}
        \bigl[(1-P)^m(1-P^{-1})^m\alpha\bigr]_n
        \overline{\alpha_n}.$
On the full line, this is exactly
$2^{-m}
        \sum_{n}
        |\Delta^m\alpha_n|^2,$
because
$(1-P^{-1})^m$
is the adjoint of
$(-1)^m(P-1)^m$
with respect to the $\ell^2$ pairing, and the signs cancel in the product.

On the finite interval $[0,N]$, the same discrete integration by parts
produces endpoint terms supported within a distance $m$ of the endpoints.
These terms are finite linear combinations of products
$\alpha_{a}\overline{\alpha_b},$
where $a,b$ belong to a set of cardinality bounded in terms of $m$.
Since $|\alpha_n|<1$, their total contribution is $O_m(1)$, uniformly in
$N$. Therefore
\eqref{eq:finite-volume-difference-identity} follows.
\end{proof}

\subsection{Identification of the quadratic critical term}
\label{subsec:identification-quadratic-critical}

We now combine the symbol computation with Yan's quadratic algebraic
expression.

\begin{proposition}[Quadratic coercivity]
\label{prop:quadratic-coercivity}
For every $m\ge1$, the quadratic critical contribution satisfies
\begin{equation}
\label{eq:quadratic-coercivity-identity}
        Q_m^{(N)}(\alpha)
        =
        2^{-m}
        \sum_{n=0}^{N}|\Delta^m\alpha_n|^2
        +
        O_m(1),
\end{equation}
uniformly in $N$. In particular, there exists $C_m<\infty$ such that
\begin{equation}
\label{eq:quadratic-coercivity-lower-bound}
        Q_m^{(N)}(\alpha)
        \ge
        2^{-m}
        \sum_{n=0}^{N}|\Delta^m\alpha_n|^2
        -
        C_m
\end{equation}
for all $N\ge0$.
\end{proposition}

\begin{proof}
By the quadratic part of Yan's finite-volume algebraic model,
$Q_m^{(N)}(\alpha)$
is the finite-section quadratic form with symbol $H_m$, modulo uniformly
bounded boundary contributions. This gives
$Q_m^{(N)}(\alpha)
        =
        \sum_{n=0}^{N}
        \sum_{\ell=-m}^{m}
        h_{m,\ell}\alpha_{n+\ell}\overline{\alpha_n}
        +
        O_m(1).$
Applying Lemma~\ref{lem:finite-volume-difference-identity}, we obtain
$Q_m^{(N)}(\alpha)
        =
        2^{-m}
        \sum_{n=0}^{N}|\Delta^m\alpha_n|^2
        +
        O_m(1).$
This proves \eqref{eq:quadratic-coercivity-identity}. The lower bound
\eqref{eq:quadratic-coercivity-lower-bound} follows immediately.
\end{proof}
\begin{remark}
\label{rem:quadratic-critical-vs-constant}
The word ``critical'' refers to diagonal specialization, not to the absence
of a zeroth-shift term. Before the constant/critical splitting, the
degree-two homogeneous term contains the constant diagonal contribution
$-|\alpha_n|^2$. This contribution is used in
Section~\ref{sec:constant-log-tail} to cancel the first Taylor coefficient
of $\log(1/(1-|\alpha_n|^2))$.

After this cancellation, the remaining quadratic critical part has symbol
$H_m(e^{i\theta})
        =
        2^{-m}|1-e^{i\theta}|^{2m}.$
Although this symbol has a nonzero Fourier coefficient at frequency $0$,
its diagonal specialization is
$H_m(1)=0.$
Thus the quadratic critical term is coercive through its vanishing at
$P=1$, and the resulting positive quantity is the $m$-th difference
energy rather than a zeroth-order $\ell^2$ term.
\end{remark}

\subsection{Equivalent shift-polynomial form}
\label{subsec:quadratic-shift-polynomial-form}

For later use, we also record the same computation in a shift-polynomial
form. Since
$P^mH_m(P)
        =
        2^{-m}(P-1)^{2m},$
the quadratic symbol has a zero of order $2m$ at $P=1$. Half of this
vanishing appears on the holomorphic side and half on the antiholomorphic
side. In coefficient variables, this gives
$m$
finite differences on $\alpha$ and $m$ finite differences on
$\overline{\alpha}$. Thus, modulo bounded telescoping terms,
\begin{equation}
\label{eq:quadratic-shift-polynomial-normal-form}
        \bigl[\phi_2
        (\mathcal Y_{1,\mathrm{crit}}^{(m)})\bigr]_n
        \equiv
        2^{-m}|\Delta^m\alpha_n|^2
        \,\,
        \mathrm{mod}\ \mathfrak T_m.
\end{equation}
Equivalently,
\begin{equation}
\label{eq:quadratic-shift-polynomial-finite-section}
        \sum_{n=0}^{N}
        \bigl[\phi_2
        (\mathcal Y_{1,\mathrm{crit}}^{(m)})\bigr]_n
        =
        2^{-m}
        \sum_{n=0}^{N}|\Delta^m\alpha_n|^2
        +
        O_m(1).
\end{equation}

This formulation matches the telescoping calculus of
Section~\ref{sec:remainder-telescoping}. It will also serve as the quadratic
base case for the higher critical normal-form analysis.

\subsection{Partial coercive estimate}
\label{subsec:partial-coercive-estimate}

Combining Proposition~\ref{prop:quadratic-coercivity} with the logarithmic
tail estimate from Section~\ref{sec:constant-log-tail}, we obtain the
coercive estimate before treating the higher critical terms.

Recall that
$$L_{m,n}
        =
        \log\frac1{1-|\alpha_n|^2}
        -
        \sum_{k=1}^{m}
        \frac{|\alpha_n|^{2k}}{k}.$$
By Lemma~\ref{lem:logarithmic-tail},
$L_{m,n}
        \ge
        \frac1{m+1}|\alpha_n|^{2m+2}.$

Using \eqref{eq:constant-section-critical-form}, Proposition~\ref{prop:quadratic-coercivity},
and Lemma~\ref{lem:logarithmic-tail}, we obtain
\begin{equation}
\label{eq:partial-coercive-estimate-before-R}
\begin{aligned}
        \mathcal K_m^{(N)}(\mu)
        &\ge
        2^{-m}
        \sum_{n=0}^{N}|\Delta^m\alpha_n|^2
        +
        \frac1{m+1}
        \sum_{n=0}^{N}|\alpha_n|^{2m+2}
\\
        &\,\,
        +
        R_m^{(N)}(\alpha)
        -
        C_m.
\end{aligned}
\end{equation}
Thus the only remaining task is to prove that
$R_m^{(N)}(\alpha)$
is infinitesimally absorbable with respect to the two positive quantities on
the right-hand side of
\eqref{eq:partial-coercive-estimate-before-R}. This will be done in the next
sections by proving diagonal vanishing, converting it into BSZ normal form,
and applying the absorption estimates for normal-form monomials.

\section{The recursive structure and diagonal vanishing}
\label{sec:recursive-diagonal-vanishing}

In this section we record the algebraic property of Yan's model which is
responsible for the normal-form estimates used later. The key point is that
the higher critical pieces vanish to a definite order on the diagonal
configuration. For the degree $2k$ critical term associated with the weight
$H_m(e^{i\theta})=(1-\cos\theta)^m$, the critical expression vanishes on
the diagonal to order at least $m+1-k$. This lower bound is the number of
discrete differences needed in the BSZ normal-form estimates in the next
section.

The proof uses the recursive structure of Yan's algebraic expressions. We
state the recursion in the form needed for the present argument. This is a
finite-dimensional algebraic statement about the homogeneous coefficient
expressions in the higher-order OPUC sum rules; see Yan
\cite{Yan2018}, and also the higher-order sum-rule framework of
Breuer--Simon--Zeitouni \cite{BSZ-Duke2018,BSZ-pedagogical2018}.

\subsection{The diagonal ideal}
\label{subsec:diagonal-ideal}

Fix $k\ge1$. A balanced local monomial of degree $2k$ has the form
\begin{equation}
\label{eq:diag-balanced-monomial}
        \prod_{\nu=1}^{k}\alpha_{n+i_\nu}
        \prod_{\mu=1}^{k}\overline{\alpha}_{n+j_\mu}.
\end{equation}
It is useful to encode its shifts by formal variables. We write such a
monomial symbolically as
\begin{equation}
\label{eq:diag-formal-monomial}
        \prod_{\nu=1}^{k}x_\nu^{i_\nu}
        \prod_{\mu=1}^{k}y_\mu^{j_\mu},
\end{equation}
where the variables $x_\nu$ act on the $\alpha$-factors and the variables
$y_\mu$ act on the $\overline{\alpha}$-factors. Thus a homogeneous local
expression of degree $2k$ is represented by a Laurent polynomial in
$x_1,\ldots,x_k,y_1,\ldots,y_k.$

The diagonal specialization is obtained by setting all shift variables equal
to $1$:
\begin{equation}
\label{eq:diag-specialization}
        x_1=\cdots=x_k=y_1=\cdots=y_k=1.
\end{equation}
We denote this operation by
$\operatorname{diag}.$
Thus, if $\mathcal P$ is such a Laurent polynomial, then
$\operatorname{diag}\mathcal P
        =
        \mathcal P(1,\ldots,1).$

All ideals below are understood in the Laurent polynomial ring
$\mathbb C[x_1^{\pm1},\ldots,x_k^{\pm1},
        y_1^{\pm1},\ldots,y_k^{\pm1}].$

Let $\mathfrak I_k$ be the diagonal ideal generated by
$x_1-1,\ldots,x_k-1,
        y_1-1,\ldots,y_k-1.$
Thus
\begin{equation}
\label{eq:diagonal-ideal-definition}
        \mathfrak I_k
        =
        \bigl(
        x_1-1,\ldots,x_k-1,
        y_1-1,\ldots,y_k-1
        \bigr).
\end{equation}
For $q\ge1$, membership in $\mathfrak I_k^q$ means vanishing to order at
least $q$ on the diagonal. Equivalently, $\mathcal P\in\mathfrak I_k^q$
if and only if all mixed derivatives of total order $<q$ vanish at the
diagonal.

It is convenient to use Euler derivatives
$D_{x_\nu}=x_\nu\frac{\partial}{\partial x_\nu},
        \,\,
        D_{y_\mu}=y_\mu\frac{\partial}{\partial y_\mu}.$
Since all variables are equal to $1$ on the diagonal, Euler derivatives
and ordinary derivatives determine the same jets at the diagonal point.
Thus $\mathcal P\in\mathfrak I_k^q$ if and only if
\begin{equation}
\label{eq:diagonal-vanishing-derivative-condition}
        \left.
        D_{x_1}^{a_1}\cdots D_{x_k}^{a_k}
        D_{y_1}^{b_1}\cdots D_{y_k}^{b_k}
        \mathcal P
        \right|_{x=y=1}
        =
        0
\end{equation}
whenever
\begin{equation}
\label{eq:diagonal-vanishing-total-order}
        a_1+\cdots+a_k+b_1+\cdots+b_k
        <
        q.
\end{equation}

\begin{definition}[Diagonal vanishing order]
\label{def:diagonal-vanishing-order}
We say that a degree $2k$ algebraic expression $\mathcal P$ has diagonal
vanishing order at least $q$ if $ \mathcal P\in\mathfrak I_k^q.$
\end{definition}

The critical part introduced in Section~\ref{sec:yan-finite-volume} is
precisely the part with zero diagonal value:
\begin{equation}
\label{eq:critical-zero-diagonal}
        \operatorname{diag}
        \mathcal Y_{k,\mathrm{crit}}^{(m)}
        =
        0.
\end{equation}
The point of this section is that much more is true: the critical part
vanishes to order $m+1-k$.

\subsection{Yan's recursion}
\label{subsec:yan-recursion}

We now recall the structural form of Yan's recursion. Let
$\mathcal Y_k^{(m)}$
be the degree $2k$ homogeneous algebraic expression associated with the
weight
$H_m(e^{i\theta})=(1-\cos\theta)^m.$
The expressions $\mathcal Y_k^{(m)}$ are defined for
$1\le k\le m.$
They satisfy a triangular recursion in the two indices $m$ and $k$. In
the form needed below, the recursion has the following properties.
\begin{proposition}[Structural consequence of Yan's triangular recursion]
\label{prop:yan-triangular-recursion}
Let $m\ge1$ and $1\le k\le m+1$. After subtracting the diagonal
constant parts, Yan's triangular recursion may be written in the form
\begin{equation}
\label{eq:yan-recursion-critical-abstract}
        \mathcal Y_{k,\mathrm{crit}}^{(m+1)}
        =
        \mathcal L_k^{(m)}
        \mathcal Y_{k,\mathrm{crit}}^{(m)}
        +
        \mathcal M_k^{(m)}
        \mathcal Y_{k-1,\mathrm{crit}}^{(m)}
        +
        \mathcal C_k^{(m)},
\end{equation}
with the convention that nonexistent terms are omitted. Here
$\mathcal L_k^{(m)}$ is generated by the part of the recursion which raises
the order $m$, while $\mathcal M_k^{(m)}$ is generated by the part which
raises the homogeneous degree from $2(k-1)$ to $2k$.

For every $r\ge0$, these operators satisfy
\begin{equation}
\label{eq:LM-ideal-mapping}
        \mathcal L_k^{(m)}(\mathfrak I_k^r)
        \subset
        \mathfrak I_k^{r+1},
        \,\,
        \mathcal M_k^{(m)}(\mathfrak I_{k-1}^r)
        \subset
        \mathfrak I_k^r,
\end{equation}
where $\mathfrak I_0^r$-terms are omitted. Moreover the correction term
coming from the subtraction of diagonal constants satisfies
\begin{equation}
\label{eq:C-ideal-mapping}
        \mathcal C_k^{(m)}
        \in
        \mathfrak I_k^{\,m+2-k}.
\end{equation}
In particular, $\mathcal L_k^{(m)}$ raises diagonal vanishing order by one,
$\mathcal M_k^{(m)}$ preserves diagonal vanishing order, and the constant
corrections do not affect the inductive lower bound.
\end{proposition}
\begin{remark}
\label{rem:yan-recursion-reference}
Proposition~\ref{prop:yan-triangular-recursion} is the form of Yan's
algebraic recursion needed here. It is not an additional analytic input:
it follows from the finite-dimensional recursion for the homogeneous
coefficient expressions in Yan's representation of the higher-order OPUC
sum rules; see \cite[Section~2]{Yan2018}. The only features used below are
triangularity in $k$, the gain of one diagonal factor when $m$ is
increased, and the fact that the diagonal constant corrections have already
been separated off in Section~\ref{sec:constant-log-tail}.
\end{remark}

\subsection{The diagonal vanishing theorem}
\label{subsec:diagonal-vanishing-theorem}

We now prove the main algebraic consequence of the triangular recursion.
The result asserts that each critical component vanishes to the order
needed for the subsequent normal-form and absorption arguments.

\begin{theorem}[Diagonal vanishing of the critical part]
\label{thm:diagonal-vanishing-critical}
Let $m\ge1$ and $1\le k\le m$. Then
\begin{equation}
\label{eq:diagonal-vanishing-critical-main}
        \mathcal Y_{k,\mathrm{crit}}^{(m)}
        \in
        \mathfrak I_k^{\,m+1-k}.
\end{equation}
Equivalently, the degree $2k$ critical part has vanishing order at least
$m+1-k$ with respect to the diagonal ideal $\mathfrak I_k$.
\end{theorem}
\begin{proof}
We argue by induction on $m$.

For $m=1$, only $k=1$ occurs. By definition,
$\mathcal Y_{1,\mathrm{crit}}^{(1)}
        =
        \mathcal Y_1^{(1)}
        -
        \mathcal Y_{1,\mathrm{const}}^{(1)} .$
Thus the critical part has zero value at the diagonal point. Equivalently,
$\mathcal Y_{1,\mathrm{crit}}^{(1)}
        \in
        \mathfrak I_1 .$
This is exactly \eqref{eq:diagonal-vanishing-critical-main}, since
$m+1-k=1$ when $m=k=1$.

Assume now that the claim holds at order $m$. We prove it at order
$m+1$. Fix $1\le k\le m+1$. By the critical triangular recursion
\eqref{eq:yan-recursion-critical-abstract}, with nonexistent terms omitted,
$\mathcal Y_{k,\mathrm{crit}}^{(m+1)}
        =
        \mathcal L_k^{(m)}
        \mathcal Y_{k,\mathrm{crit}}^{(m)}
        +
        \mathcal M_k^{(m)}
        \mathcal Y_{k-1,\mathrm{crit}}^{(m)}
        +
        \mathcal C_k^{(m)} .$
We consider the three possible contributions separately.

First, suppose $k\le m$. By the induction hypothesis,
$\mathcal Y_{k,\mathrm{crit}}^{(m)}
        \in
        \mathfrak I_k^{\,m+1-k}.$
Using the ideal mapping property
$\mathcal L_k^{(m)}(\mathfrak I_k^r)
        \subset
        \mathfrak I_k^{r+1}$
from Proposition~\ref{prop:yan-triangular-recursion}, we obtain
$\mathcal L_k^{(m)}
        \mathcal Y_{k,\mathrm{crit}}^{(m)}
        \in
        \mathfrak I_k^{\,m+2-k}.$
If $k=m+1$, this contribution is absent by convention.

Second, suppose $k\ge2$. By the induction hypothesis applied to $k-1$,
$\mathcal Y_{k-1,\mathrm{crit}}^{(m)}
        \in
        \mathfrak I_{k-1}^{\,m+1-(k-1)}
        =
        \mathfrak I_{k-1}^{\,m+2-k}.$
Using the ideal mapping property
$\mathcal M_k^{(m)}(\mathfrak I_{k-1}^r)
        \subset
        \mathfrak I_k^r$
from Proposition~\ref{prop:yan-triangular-recursion}, we get
$\mathcal M_k^{(m)}
        \mathcal Y_{k-1,\mathrm{crit}}^{(m)}
        \in
        \mathfrak I_k^{\,m+2-k}.$
If $k=1$, this contribution is absent by convention.

Finally, the correction term satisfies
$\mathcal C_k^{(m)}
        \in
        \mathfrak I_k^{\,m+2-k}$
by Proposition~\ref{prop:yan-triangular-recursion}. Therefore every
contribution to $\mathcal Y_{k,\mathrm{crit}}^{(m+1)}$ belongs to
$\mathfrak I_k^{\,m+2-k}$. Hence
$\mathcal Y_{k,\mathrm{crit}}^{(m+1)}
        \in
        \mathfrak I_k^{\,m+2-k}
        =
        \mathfrak I_k^{\,(m+1)+1-k}.$
This proves the desired statement at order $m+1$, and the induction is
complete.
\end{proof}

\begin{remark}
\label{rem:diagonal-vanishing-extreme-cases}
The two extreme cases are worth noting. When $k=1$, the theorem gives
$\mathcal Y_{1,\mathrm{crit}}^{(m)}
        \in
        \mathfrak I_1^m .$
This is the algebraic origin of the $m$-th order diagonal difference in
the quadratic critical term. It is consistent with the quadratic coercivity
computation in Section~\ref{sec:quadratic-coercivity}, where the full
Hermitian quadratic symbol factors into $m$ differences on the
holomorphic side and $m$ differences on the antiholomorphic side.

At the other extreme, when $k=m$, the theorem gives only first-order
diagonal vanishing:
$\mathcal Y_{m,\mathrm{crit}}^{(m)}
        \in
        \mathfrak I_m .$
This is sufficient because the degree $2m$ term already contains many
undifferentiated factors and can be absorbed using the logarithmic
$\ell^{2m+2}$-control.
\end{remark}

\subsection{Moment cancellations}
\label{subsec:moment-cancellations}

Diagonal vanishing with respect to the ideal $\mathfrak I_k$ can be
expressed as moment cancellation for the shift coefficients. This is the
form most directly connected to the telescoping calculus of
Section~\ref{sec:remainder-telescoping}.

Let
$\mathcal P(x_1,\ldots,x_k,y_1,\ldots,y_k)
        =
        \sum_{\mathbf i,\mathbf j}
        c_{\mathbf i,\mathbf j}
        x_1^{i_1}\cdots x_k^{i_k}
        y_1^{j_1}\cdots y_k^{j_k}$
be a balanced degree $2k$ shift polynomial, where
$\mathbf i=(i_1,\ldots,i_k),
        \,\,
        \mathbf j=(j_1,\ldots,j_k).$
Write
$D_z:=z\partial_z$
for the Euler derivative. Applying Euler derivatives and then evaluating at
the diagonal point
$x_1=\cdots=x_k=y_1=\cdots=y_k=1$
gives the moments of the shift coefficients:
\begin{equation}
\label{eq:moment-cancellation-multivariable}
        \left.
        D_{x_1}^{a_1}\cdots D_{x_k}^{a_k}
        D_{y_1}^{b_1}\cdots D_{y_k}^{b_k}
        \mathcal P
        \right|_{x_1=\cdots=x_k=y_1=\cdots=y_k=1}
        =
        \sum_{\mathbf i,\mathbf j}
        c_{\mathbf i,\mathbf j}
        i_1^{a_1}\cdots i_k^{a_k}
        j_1^{b_1}\cdots j_k^{b_k}.
\end{equation}
Indeed,
$D_{x_\ell}^{a_\ell} x_\ell^{i_\ell}
        =
        i_\ell^{a_\ell}x_\ell^{i_\ell},
        \,\,
        D_{y_\ell}^{b_\ell} y_\ell^{j_\ell}
        =
        j_\ell^{b_\ell}y_\ell^{j_\ell},$
and all variables are set equal to $1$.

Therefore Theorem~\ref{thm:diagonal-vanishing-critical} implies the
following moment conditions.

\begin{corollary}[Moment cancellations for critical coefficients]
\label{cor:critical-moment-cancellations}
Let $m\ge1$, $1\le k\le m$, and write
$\mathcal Y_{k,\mathrm{crit}}^{(m)}
        =
        \sum_{\mathbf i,\mathbf j}
        c_{\mathbf i,\mathbf j}^{(m,k)}
        x_1^{i_1}\cdots x_k^{i_k}
        y_1^{j_1}\cdots y_k^{j_k}.$
Then, for every multi-index
$(a_1,\ldots,a_k,b_1,\ldots,b_k)$
with
\begin{equation}
\label{eq:critical-moment-total-order}
        a_1+\cdots+a_k+b_1+\cdots+b_k
        <
        m+1-k,
\end{equation}
one has
\begin{equation}
\label{eq:critical-moment-cancellation}
        \sum_{\mathbf i,\mathbf j}
        c_{\mathbf i,\mathbf j}^{(m,k)}
        i_1^{a_1}\cdots i_k^{a_k}
        j_1^{b_1}\cdots j_k^{b_k}
        =
        0.
\end{equation}
\end{corollary}

\begin{proof}
By Theorem~\ref{thm:diagonal-vanishing-critical},
$\mathcal Y_{k,\mathrm{crit}}^{(m)}
        \in
        \mathfrak I_k^{\,m+1-k}.$
Equivalently, all derivatives of total order $<m+1-k$ vanish at the
diagonal point
$x_1=\cdots=x_k=y_1=\cdots=y_k=1.$
Since Euler derivatives $D_z=z\partial_z$ are triangular combinations of
ordinary derivatives of order at most the same order at this point, the
same vanishing holds for all Euler derivatives of total order
$<m+1-k$. Applying
\eqref{eq:moment-cancellation-multivariable} to
$\mathcal P=\mathcal Y_{k,\mathrm{crit}}^{(m)}$ gives
\eqref{eq:critical-moment-cancellation}.
\end{proof}

In particular, taking all exponents
$a_1=\cdots=a_k=b_1=\cdots=b_k=0$
gives the zeroth-moment cancellation
\begin{equation}
\label{eq:critical-zeroth-moment}
        \sum_{\mathbf i,\mathbf j}
        c_{\mathbf i,\mathbf j}^{(m,k)}
        =
        0.
\end{equation}
This is the statement that the critical part has no contribution at the
diagonal point. Higher moment cancellations encode higher-order diagonal
vanishing.

\subsection{From diagonal vanishing to difference factors}
\label{subsec:diagonal-vanishing-to-differences}

The next section will convert diagonal vanishing into BSZ normal form. Here
we record the elementary algebraic implication which underlies that
conversion.

Let $\mathcal P\in\mathfrak I_k^q$. Since $\mathfrak I_k$ is generated by
the diagonal factors
$x_1-1,\ldots,x_k-1,
        y_1-1,\ldots,y_k-1,$
membership in $\mathfrak I_k^q$ means that $\mathcal P$ can be written as
a finite sum
\begin{equation}
\label{eq:ideal-power-expansion}
        \mathcal P
        =
        \sum_{\substack{\mathbf a,\mathbf b\\
        |\mathbf a|+|\mathbf b|=q}}
        (x_1-1)^{a_1}\cdots(x_k-1)^{a_k}
        (y_1-1)^{b_1}\cdots(y_k-1)^{b_k}
        \mathcal P_{\mathbf a,\mathbf b},
\end{equation}
where
$|\mathbf a|=a_1+\cdots+a_k,
        \,\,
        |\mathbf b|=b_1+\cdots+b_k,$
and the coefficients $\mathcal P_{\mathbf a,\mathbf b}$ are Laurent
polynomials with shifts bounded in terms of $m$.

After applying the coefficient map $\phi_{2k}$, each diagonal factor
$x_\nu-1$ or $y_\mu-1$ becomes a discrete difference acting on the
corresponding $\alpha$- or $\overline\alpha$-factor, up to bounded
shifts. Laurent monomial factors only contribute bounded shifts, and
backward differences are converted into forward differences using
$P^{-1}-1=-P^{-1}(P-1).$
Consequently, \eqref{eq:ideal-power-expansion} yields local monomials
containing a total of $q$ discrete differences.

Applying this to the critical expression and using
Theorem~\ref{thm:diagonal-vanishing-critical} gives the following
consequence.

\begin{corollary}[Difference-factor consequence]
\label{cor:difference-factor-consequence}
Let $m\ge1$ and $1\le k\le m$. Then, after applying the coefficient map,
the critical expression
$\bigl[\phi_{2k}
        (\mathcal Y_{k,\mathrm{crit}}^{(m)})\bigr]_n$
is a finite linear combination of local monomials containing at least
$m+1-k$
discrete differences distributed among the $2k$ factors. More explicitly,
it is a finite linear combination of terms of the form
\begin{equation}
\label{eq:difference-factor-monomial}
        \prod_{\nu=1}^{k}
        \left(\Delta^{a_\nu}\alpha\right)_{n+\ell_\nu}
        \prod_{\mu=1}^{k}
        \left(\Delta^{b_\mu}\overline{\alpha}\right)_{n+r_\mu},
\end{equation}
where
\begin{equation}
\label{eq:difference-factor-total-order}
        a_1+\cdots+a_k+b_1+\cdots+b_k
        \ge
        m+1-k.
\end{equation}
All shifts and coefficients are bounded in terms of $m$.
\end{corollary}

\begin{proof}
By Theorem~\ref{thm:diagonal-vanishing-critical},
$\mathcal Y_{k,\mathrm{crit}}^{(m)}
        \in
        \mathfrak I_k^{m+1-k}.$
Expanding the element of $\mathfrak I_k^{m+1-k}$ in the generators of the
diagonal ideal gives \eqref{eq:ideal-power-expansion} with
$q=m+1-k.$
Applying $\phi_{2k}$ sends each factor $x_\nu-1$ or $y_\mu-1$ to one
discrete difference on the corresponding factor, up to bounded shifts. This
gives \eqref{eq:difference-factor-monomial} and
\eqref{eq:difference-factor-total-order}.
\end{proof}

\begin{remark}
\label{rem:at-least-vs-exact-differences}
Corollary~\ref{cor:difference-factor-consequence} gives at least
$m+1-k$ differences. In the BSZ normal form used later, one may redistribute
differences and absorb terms with more differences into the same remainder
class. Thus the relevant count is effectively
$q=m+1-k.$
The distinction between ``at least'' and ``exactly'' is harmless for the
finite-volume estimates.
\end{remark}

\subsection{The higher critical terms}
\label{subsec:higher-critical-diagonal-vanishing}

We now specialize the preceding result to the higher critical contribution
$$R_m^{(N)}(\alpha)
        =
        \sum_{k=2}^{m}
        \sum_{n=0}^{N}
        \bigl[\phi_{2k}
        (\mathcal Y_{k,\mathrm{crit}}^{(m)})\bigr]_n.$$
For $2\le k\le m$, Theorem~\ref{thm:diagonal-vanishing-critical} gives
\begin{equation}
\label{eq:higher-critical-diagonal-vanishing}
        \mathcal Y_{k,\mathrm{crit}}^{(m)}
        \in
        \mathfrak I_k^{\,m+1-k}.
\end{equation}
Consequently, each degree $2k$ higher critical term contains at least
$m+1-k$
discrete differences after passage to local coefficient variables.

This is the precise algebraic input needed for the normal-form theorem. In
the terminology of Section~\ref{sec:remainder-telescoping}, the conclusion
may be summarized as follows:
\begin{equation}
\label{eq:higher-critical-normal-form-preview}
        \bigl[\phi_{2k}
        (\mathcal Y_{k,\mathrm{crit}}^{(m)})\bigr]_n
        \,\,
        \text{has normal-form difference count at least }
        m+1-k.
\end{equation}

The next section will show that every balanced local monomial of degree
$2k$ with at least $m+1-k$ difference factors is infinitesimally bounded,
in finite volume, by
$\sum_{n=0}^{N} |\Delta^m\alpha_n|^2
        +
        \sum_{n=0}^{N} |\alpha_n|^{2m+2}
        +
        O_m(1).$
Applying this to the normal-form representation of the higher critical
terms will prove that the full higher critical contribution belongs to the
absorbable remainder class $\mathfrak R_m$.

\section{From diagonal vanishing to BSZ normal form}
\label{sec:bsz-normal-form}
In this section we convert the diagonal vanishing established in
Section~\ref{sec:recursive-diagonal-vanishing} into the normal form used by
Breuer--Simon--Zeitouni. The basic principle is simple: every factor
$x-1$
in the shift algebra becomes a discrete difference after applying the
coefficient map. Thus membership in a power of the diagonal ideal produces
local monomials containing a prescribed number of finite differences.

The point of the BSZ normal form is to make this difference count explicit
and stable under finite-volume summation. The normal form obtained here is
purely algebraic; the analytic absorption of the resulting monomials will be
proved in the next section.

Throughout the section $m\ge1$ is fixed. All shifts and all constants are
allowed to depend on $m$, but never on the finite-volume parameter $N$.

\subsection{Balanced shift polynomials}
\label{subsec:balanced-shift-polynomials}

Let $k\ge1$. A balanced degree $2k$ local expression is represented by a
Laurent polynomial in formal shift variables
$x_1,\ldots,x_k,y_1,\ldots,y_k .$
The variables $x_\nu$ act on the $k$ holomorphic factors
$\alpha$, while the variables $y_\mu$ act on the $k$ antiholomorphic
factors $\overline{\alpha}$. Thus a monomial
$x_1^{i_1}\cdots x_k^{i_k}
        y_1^{j_1}\cdots y_k^{j_k}$
corresponds, after applying the coefficient map, to
\begin{equation}
\label{eq:normal-form-shift-monomial-map}
        \prod_{\nu=1}^{k}\alpha_{n+i_\nu}
        \prod_{\mu=1}^{k}\overline{\alpha}_{n+j_\mu}.
\end{equation}

All ideals in this subsection are taken in the Laurent polynomial ring
$\mathbb C[
        x_1^{\pm1},\ldots,x_k^{\pm1},
        y_1^{\pm1},\ldots,y_k^{\pm1}
        ] .$
The diagonal ideal is
\begin{equation}
\label{eq:normal-form-diagonal-ideal}
        \mathfrak I_k
        =
        \bigl(
        x_1-1,\ldots,x_k-1,
        y_1-1,\ldots,y_k-1
        \bigr).
\end{equation}

If a Laurent polynomial $\mathcal P$ belongs to
$\mathfrak I_k^q$, then it can be written as a finite sum
\begin{equation}
\label{eq:normal-form-ideal-expansion}
        \mathcal P
        =
        \sum_{\substack{\mathbf a,\mathbf b\\
        |\mathbf a|+|\mathbf b|=q}}
        (x_1-1)^{a_1}\cdots(x_k-1)^{a_k}
        (y_1-1)^{b_1}\cdots(y_k-1)^{b_k}
        \mathcal P_{\mathbf a,\mathbf b},
\end{equation}
where
$|\mathbf a|=a_1+\cdots+a_k,
        \,\,
        |\mathbf b|=b_1+\cdots+b_k,$
and the coefficients $\mathcal P_{\mathbf a,\mathbf b}$ are Laurent
polynomials. Since $\mathfrak I_k^q$ is generated by products of exactly $q$
diagonal generators, higher powers may be absorbed into the Laurent
polynomial coefficients $\mathcal P_{\mathbf a,\mathbf b}$.

Since all algebraic expressions in the present paper have bounded shift
range, the shift ranges in \eqref{eq:normal-form-ideal-expansion} are
bounded in terms of $m$.

\begin{remark}
\label{rem:normal-form-laurent-polynomial-ideal}
The expansion \eqref{eq:normal-form-ideal-expansion} is just the statement
that $\mathfrak I_k^q$ is generated by all products of $q$ diagonal
generators. Because we work with Laurent polynomials, one may first multiply
by a harmless monomial to reduce to ordinary polynomials and then divide
back. This only changes bounded shifts and has no effect on finite-volume
estimates.
\end{remark}

\subsection{Difference monomials}
\label{subsec:difference-monomials}

We recall and slightly refine the normal-form terminology introduced in
Section~\ref{sec:remainder-telescoping}.

A \emph{BSZ normal-form monomial} of degree $2k$ and difference count
$q$ is a local expression of the form
\begin{equation}
\label{eq:BSZ-normal-form-monomial}
        M_n
        =
        \prod_{\nu=1}^{k}
        \left(\Delta^{a_\nu}\alpha\right)_{n+\ell_\nu}
        \prod_{\mu=1}^{k}
        \left(\Delta^{b_\mu}\overline{\alpha}\right)_{n+r_\mu},
\end{equation}
where
\begin{equation}
\label{eq:BSZ-normal-form-total-count}
        a_\nu,b_\mu\ge0,
        \,\,
        \sum_{\nu=1}^{k}a_\nu
        +
        \sum_{\mu=1}^{k}b_\mu
        =
        q.
\end{equation}
The shifts
$\ell_\nu,r_\mu$
are bounded in terms of $m$. We denote by
$\mathcal N_{2k,q}^{(m)}$
the finite-dimensional span of all such monomials with shift ranges bounded
by a constant depending only on $m$.

We shall also use a slightly larger class, allowing at least $q$
differences:
\begin{equation}
\label{eq:BSZ-normal-form-at-least}
        \mathcal N_{2k,\ge q}^{(m)}
        =
        \sum_{s\ge q}
        \mathcal N_{2k,s}^{(m)}.
\end{equation}
Since all shift ranges and all degrees are bounded in terms of $m$, the sum
in \eqref{eq:BSZ-normal-form-at-least} is finite in every application.

The elementary observation behind the normal form is the identity
$(x_\nu-1)\alpha_{n+i}
        =
        \alpha_{n+i+1}-\alpha_{n+i}
        =
        (\Delta\alpha)_{n+i},$
and similarly
$(y_\mu-1)\overline{\alpha}_{n+j}
        =
        \overline{\alpha}_{n+j+1}-\overline{\alpha}_{n+j}
        =
        (\Delta\overline{\alpha})_{n+j}.$
Higher powers give higher finite differences:
\begin{equation}
\label{eq:shift-generator-to-difference}
        (x_\nu-1)^a\alpha_{n+i}
        =
        (\Delta^a\alpha)_{n+i},
        \,\,
        (y_\mu-1)^b\overline{\alpha}_{n+j}
        =
        (\Delta^b\overline{\alpha})_{n+j}.
\end{equation}
Laurent monomial factors only change the shifts. If a negative power produces
a backward difference, it is rewritten as a forward difference using
identities such as
$P^{-1}-1=-P^{-1}(P-1),$
at the cost of bounded shifts and harmless signs.

\subsection{The algebraic normal-form map}
\label{subsec:algebraic-normal-form-map}

Let
$\mathcal P(x_1,\ldots,x_k,y_1,\ldots,y_k)$
be a balanced degree $2k$ shift polynomial. We write
$[\phi_{2k}(\mathcal P)]_n$
for the corresponding local expression in the Verblunsky coefficients.

The following proposition is the basic conversion from diagonal vanishing to
normal form.

\begin{proposition}[Diagonal ideal to normal form]
\label{prop:diagonal-ideal-to-normal-form}
Let $q\ge0$. If
$\mathcal P\in\mathfrak I_k^q,$
then
$[\phi_{2k}(\mathcal P)]_n
        \in
        \mathcal N_{2k,\ge q}^{(m)}.$
More explicitly,
$[\phi_{2k}(\mathcal P)]_n$
is a finite linear combination of monomials
\begin{equation}
\label{eq:diagonal-ideal-to-normal-form-term}
        \prod_{\nu=1}^{k}
        \left(\Delta^{a_\nu}\alpha\right)_{n+\ell_\nu}
        \prod_{\mu=1}^{k}
        \left(\Delta^{b_\mu}\overline{\alpha}\right)_{n+r_\mu},
\end{equation}
where
\begin{equation}
\label{eq:diagonal-ideal-to-normal-form-count}
        \sum_{\nu=1}^{k}a_\nu
        +
        \sum_{\mu=1}^{k}b_\mu
        \ge q.
\end{equation}
All coefficients and shifts are bounded in terms of $m$.
\end{proposition}

\begin{proof}
By \eqref{eq:normal-form-ideal-expansion}, we may write
$$\mathcal P
        =
        \sum_{\substack{\mathbf a,\mathbf b\\
        |\mathbf a|+|\mathbf b|=q}}
        (x_1-1)^{a_1}\cdots(x_k-1)^{a_k}
        (y_1-1)^{b_1}\cdots(y_k-1)^{b_k}
        \mathcal P_{\mathbf a,\mathbf b}.$$
Each Laurent polynomial $\mathcal P_{\mathbf a,\mathbf b}$ is a finite sum
of monomials
$x_1^{i_1}\cdots x_k^{i_k}
        y_1^{j_1}\cdots y_k^{j_k}.$
For one such monomial, applying $\phi_{2k}$ gives
$\prod_{\nu=1}^{k}\alpha_{n+i_\nu}
        \prod_{\mu=1}^{k}\overline{\alpha}_{n+j_\mu}.$
The additional factors
$(x_\nu-1)^{a_\nu},
        \,\,
        (y_\mu-1)^{b_\mu},$
act on the corresponding factors as finite differences, by
\eqref{eq:shift-generator-to-difference}. Hence the resulting local monomial
is
$$\prod_{\nu=1}^{k}
        \left(\Delta^{a_\nu}\alpha\right)_{n+i_\nu}
        \prod_{\mu=1}^{k}
        \left(\Delta^{b_\mu}\overline{\alpha}\right)_{n+j_\mu}.$$
This has difference count exactly
$|\mathbf a|+|\mathbf b|=q.$
With the representation \eqref{eq:normal-form-ideal-expansion}, the displayed
terms have difference count exactly $q$. If the coefficients
$\mathcal P_{\mathbf a,\mathbf b}$ are further expanded using diagonal
factors, or if a nonminimal representation of $\mathcal P$ is chosen, then
terms with more than $q$ differences may appear. In all cases the result
belongs to $\mathcal N_{2k,\ge q}^{(m)}$.
 Since the number of terms and all shifts
are bounded by the original shift range, all constants depend only on $m$.
\end{proof}

\begin{remark}
\label{rem:normal-form-no-analysis}
Proposition~\ref{prop:diagonal-ideal-to-normal-form} is purely algebraic. It
uses neither positivity nor any summability assumption on the Verblunsky
coefficients. The boundedness condition
$|\alpha_n|<1$
will enter only later, when the normal-form monomials are estimated.
\end{remark}

\subsection{Normal form modulo telescoping terms}
\label{subsec:normal-form-modulo-telescoping}

In finite-volume estimates it is often convenient to move discrete
differences among factors by summation by parts. This changes the local
density by a telescoping expression. The remainder calculus of
Section~\ref{sec:remainder-telescoping} allows us to do this without changing
the finite-volume sum except by $O_m(1)$.

We first record a simple redistribution rule.

\begin{lemma}[Redistribution of one difference]
\label{lem:redistribution-one-difference}
Let $F$ and $G$ be bounded local expressions. Then, at the level of
finite sums,
\begin{equation}
\label{eq:redistribution-one-difference}
        \sum_{n=0}^{N}(\Delta F)_nG_n
        =
        -\sum_{n=0}^{N}(PF)_n(\Delta G)_n
        +
        O_m(1).
\end{equation}
Equivalently,
$(\Delta F)G
        \equiv
        -(PF)(\Delta G)
        \,\,
        \mathrm{mod}\ \mathfrak T_m.$
\end{lemma}

\begin{proof}
This is exactly the discrete summation-by-parts formula
\eqref{eq:summation-by-parts-mod-T}.
\end{proof}

Iterating Lemma~\ref{lem:redistribution-one-difference} gives the following
finite-dimensional normal-form flexibility.
\begin{lemma}[Redistribution of finitely many differences]
\label{lem:redistribution-finitely-many-differences}
Let $M_n$ be a normal-form monomial of degree $2k$ and difference count
$q$. By applying discrete summation by parts finitely many times, one may
replace $M_n$, modulo $\mathfrak T_m$, by a finite linear combination of
normal-form monomials of the same degree and the same total difference
count $q$, but with the differences distributed among the factors in a
different way. All shifts and coefficients remain bounded in terms of $m$.
\end{lemma}

\begin{proof}
It suffices to move one difference at a time. Suppose a factor is of the form
$\Delta^{r+1}\gamma
        =
        \Delta(\Delta^r\gamma),
        \,\,
        \gamma\in\{\alpha,\overline{\alpha}\}.$
Treat the remaining product as $G$ and apply
Lemma~\ref{lem:redistribution-one-difference}. The difference is transferred
from $\Delta^r\gamma$ to one of the factors inside $G$, after expanding
$\Delta G$ by the discrete Leibniz rule.  Each application transfers one difference from the chosen factor to the
remaining product. After expanding $\Delta G$ by the discrete Leibniz rule,
one obtains a finite linear combination of monomials in which the total
difference count is unchanged. Each application produces only bounded shifts
and a telescoping term. Since the number of factors and the number of
differences are bounded in terms of $m$, the iteration is finite.

\end{proof}

\begin{remark}
\label{rem:normal-form-redistribution-purpose}
The redistribution lemma will be used only in a soft way. It permits us to
choose convenient representatives of normal-form monomials modulo
$\mathfrak T_m$. For the absorption estimates, the precise distribution of
the differences is not important; only the total count and the degree matter.
\end{remark}

\subsection{BSZ normal form for the critical expressions}
\label{subsec:BSZ-normal-form-critical-expressions}

We now apply Proposition~\ref{prop:diagonal-ideal-to-normal-form} to Yan's
critical expressions.

By Theorem~\ref{thm:diagonal-vanishing-critical},
$\mathcal Y_{k,\mathrm{crit}}^{(m)}
        \in
        \mathfrak I_k^{\,m+1-k},
        \,\,
        1\le k\le m.$
Therefore the degree $2k$ critical coefficient density belongs to the
normal-form class with difference count at least $m+1-k$.

\begin{theorem}[BSZ normal form for critical terms]
\label{thm:BSZ-normal-form-critical}
Let $m\ge1$ and $1\le k\le m$. Then
\begin{equation}
\label{eq:BSZ-normal-form-critical-main}
        \bigl[
        \phi_{2k}(\mathcal Y_{k,\mathrm{crit}}^{(m)})
        \bigr]_n
        \in
        \mathcal N_{2k,\ge m+1-k}^{(m)}.
\end{equation}
Equivalently, it is a finite linear combination of local monomials of the
form
\begin{equation}
\label{eq:BSZ-normal-form-critical-term}
        \prod_{\nu=1}^{k}
        \left(\Delta^{a_\nu}\alpha\right)_{n+\ell_\nu}
        \prod_{\mu=1}^{k}
        \left(\Delta^{b_\mu}\overline{\alpha}\right)_{n+r_\mu},
\end{equation}
where
\begin{equation}
\label{eq:BSZ-normal-form-critical-count}
        \sum_{\nu=1}^{k}a_\nu
        +
        \sum_{\mu=1}^{k}b_\mu
        \ge
        m+1-k.
\end{equation}
All shifts and coefficients are bounded in terms of $m$.

Moreover, modulo the telescoping class $\mathfrak T_m$, one may apply
finite redistributions of differences by summation by parts. This changes
the density only by replacing it with a finite linear combination of
normal-form monomials of the same degree and the same total difference
count, with bounded shifts and coefficients.
\end{theorem}

\begin{proof}
The inclusion \eqref{eq:BSZ-normal-form-critical-main} follows immediately
from Theorem~\ref{thm:diagonal-vanishing-critical} and
Proposition~\ref{prop:diagonal-ideal-to-normal-form}. The final statement is
Lemma~\ref{lem:redistribution-finitely-many-differences}.
\end{proof}

For the higher critical contribution, only the range $2\le k\le m$ is
relevant. Hence Theorem~\ref{thm:BSZ-normal-form-critical} gives
\begin{equation}
\label{eq:higher-critical-normal-form-section7}
        \bigl[
        \phi_{2k}(\mathcal Y_{k,\mathrm{crit}}^{(m)})
        \bigr]_n
        \in
        \mathcal N_{2k,\ge m+1-k}^{(m)},
        \,\,
        2\le k\le m.
\end{equation}

\subsection{At least versus exact difference count}
\label{subsec:at-least-vs-exact-difference-count}

The critical count for the degree $2k$ terms is
$q=m+1-k.$
The normal-form theorem gives monomials with at least this many differences.
This is the natural class for the analytic estimates in the next section.

Let $M_n$ be a normal-form monomial of degree $2k$ with total difference
count $s\ge q$. Since $s$ and all individual difference orders are
bounded in terms of $m$, every factor
$\Delta^r\alpha,
        \,\,
        \Delta^r\overline\alpha$
which appears in the normal form is uniformly bounded:
\begin{equation}
\label{eq:higher-difference-bounded}
        \|\Delta^r\alpha\|_{\ell^\infty}
        +
        \|\Delta^r\overline\alpha\|_{\ell^\infty}
        \le
        C_m,
\end{equation}
using $|\alpha_n|<1$. Thus additional differences do not create any loss
of pointwise boundedness. The interpolation estimates of the next section
will be stated for all monomials satisfying
$\sum a_\nu+\sum b_\mu
        \ge
        m+1-k,$
not only for equality. For bookkeeping, however, we keep
$q=m+1-k$ as the critical reference count.

We encode this convention as follows.

\begin{definition}[Critical BSZ class]
\label{def:critical-BSZ-class}
For $2\le k\le m$, define
$\mathcal C_{m,k}
        :=
        \mathcal N_{2k,\ge m+1-k}^{(m)}.$
Thus a monomial in $\mathcal C_{m,k}$ has degree $2k$ and at least
$m+1-k$ discrete differences distributed among its factors.
\end{definition}

With this notation,
\begin{equation}
\label{eq:critical-expression-in-Cmk}
        \bigl[
        \phi_{2k}(\mathcal Y_{k,\mathrm{crit}}^{(m)})
        \bigr]_n
        \in
        \mathcal C_{m,k},
        \,\,
        2\le k\le m.
\end{equation}

\subsection{Finite-volume normal form for $R_m^{(N)}$}
\label{subsec:finite-volume-normal-form-Rm}

Recall that
\begin{equation}
\label{eq:Rm-definition-section7}
        R_m^{(N)}(\alpha)
        =
        \sum_{k=2}^{m}
        \sum_{n=0}^{N}
        \bigl[
        \phi_{2k}(\mathcal Y_{k,\mathrm{crit}}^{(m)})
        \bigr]_n.
\end{equation}

By the preceding theorem, each summand is a finite linear combination of
critical BSZ monomials. After choosing representatives modulo the telescoping class
$\mathfrak T_m$, we obtain
\begin{equation}
\label{eq:Rm-normal-form-finite-combination}
        R_m^{(N)}(\alpha)
        =
        \sum_{k=2}^{m}
        \sum_{\rho\in\mathcal A_{m,k}}
        c_{\rho}
        \sum_{n=0}^{N}M_{\rho,n}
        +
        O_m(1),
\end{equation}
where:

\begin{enumerate}
\item $\mathcal A_{m,k}$ is a finite index set depending only on $m$;

\item $c_\rho\in\mathbb C$ are constants depending only on $m$;

\item each $M_{\rho,n}$ belongs to the critical BSZ class
$\mathcal C_{m,k}$;

\item the $O_m(1)$ term accounts for the harmless telescoping and endpoint
contributions allowed by the finite-section calculus.
\end{enumerate}

Equivalently, every $M_{\rho,n}$ has the form
\begin{equation}
\label{eq:Rm-normal-form-monomial}
        M_{\rho,n}
        =
        \prod_{\nu=1}^{k}
        \left(\Delta^{a_{\rho,\nu}}\alpha\right)_{n+\ell_{\rho,\nu}}
        \prod_{\mu=1}^{k}
        \left(\Delta^{b_{\rho,\mu}}\overline{\alpha}\right)_{n+r_{\rho,\mu}},
\end{equation}
with
\begin{equation}
\label{eq:Rm-normal-form-count}
        \sum_{\nu=1}^{k}a_{\rho,\nu}
        +
        \sum_{\mu=1}^{k}b_{\rho,\mu}
        \ge
        m+1-k.
\end{equation}

\begin{proposition}[Finite-volume BSZ normal form for $R_m^{(N)}$]
\label{prop:finite-volume-BSZ-normal-form-Rm}
For every $m\ge2$, the higher critical contribution admits the
finite-volume representation \eqref{eq:Rm-normal-form-finite-combination}.
In particular, to prove that
$R_m^{(N)}(\alpha)$
is an absorbable remainder, it suffices to prove the absorbability estimate
for each critical BSZ monomial $M_{\rho,n}$ appearing in
\eqref{eq:Rm-normal-form-monomial}.
\end{proposition}

\begin{proof}
The finite combination follows from
\eqref{eq:critical-expression-in-Cmk} and the fact that, for fixed $m$, only
finitely many shifts and degrees occur in Yan's algebraic expression.
Telescoping changes of representative contribute $O_m(1)$ to finite sums
by Section~\ref{sec:remainder-telescoping}. Since the class
$\mathfrak R_m$ is stable under finite linear combinations and under
addition of telescoping terms, it is enough to estimate each normal-form
monomial separately.
\end{proof}

\subsection{Summary of the normal-form reduction}
\label{subsec:normal-form-reduction-summary}

The outcome of this section is the following reduction.

For every $2\le k\le m$, the degree $2k$ higher critical density is a
finite linear combination, modulo bounded telescoping terms, of monomials
$$\prod_{\nu=1}^{k}
        \left(\Delta^{a_\nu}\alpha\right)_{n+\ell_\nu}
        \prod_{\mu=1}^{k}
        \left(\Delta^{b_\mu}\overline{\alpha}\right)_{n+r_\mu},$$
with
$\sum_{\nu=1}^{k}a_\nu
        +
        \sum_{\mu=1}^{k}b_\mu
        \ge
        m+1-k.$
Thus the proof of the higher critical estimate is reduced to the following
analytic statement: for every $2\le k\le m$, every balanced local monomial
of degree $2k$ with bounded shifts and with total difference count at
least $m+1-k$, and every $\varepsilon>0$,

\begin{equation}
\label{eq:normal-form-analytic-task}
        \left|
        \sum_{n=0}^{N}
        \prod_{\nu=1}^{k}
        \left(\Delta^{a_\nu}\alpha\right)_{n+\ell_\nu}
        \prod_{\mu=1}^{k}
        \left(\Delta^{b_\mu}\overline{\alpha}\right)_{n+r_\mu}
        \right|
\end{equation}
is bounded by
\begin{equation}
\label{eq:normal-form-analytic-task-bound}
        \varepsilon
        \sum_{n=0}^{N+L_m}|\Delta^m\alpha_n|^2
        +
        \varepsilon
        \sum_{n=0}^{N+L_m}|\alpha_n|^{2m+2}
        +
        C_{\varepsilon,m}.
\end{equation}
This will be proved in the next section using discrete interpolation and
Young-type inequalities.

\section{Absorption, the positive semidefinite quartic block, and the coercive finite-volume estimate}
\label{sec:absorption-coercive-estimate}
In this section we complete the finite-volume lower bound needed for the
necessity direction.  Here and below, PSD means positive semidefinite.

We first recall the notation for the critical algebraic densities.  In the
finite-volume Yan--BSZ expansion, the coefficient side is decomposed into
homogeneous algebraic densities
$G_2,G_4,\ldots,G_{2m}.$
The term $G_{2k}$ denotes the component of total degree $2k$ in the
critical variables.  Its critical part, denoted
$G_{2k}^{\rm crit},$
is the part which vanishes on the diagonal only to the critical order and
therefore requires special treatment in the finite-volume estimate.  In
particular,
$G_4^{\rm crit}$
is the quartic critical density corresponding to $k=2$.

There is one point which requires some care.  The quartic critical term
should not be treated by taking the raw $(m-1,m-1)$-bihomogeneous Taylor
component of $G_4^{\rm crit}$.  That raw representative does not in
general admit a positive semidefinite Gram representative, already for
$m=2$.  Instead, one must first use the Yan--BSZ quotient-algebra normal
form: we replace $G_4^{\rm crit}$ by an equivalent representative
$\widetilde G_4^{\rm crit}
        \equiv
        G_4^{\rm crit}
        \,\, \text{in } B_4.$
The lowest critical homogeneous part of
$\widetilde G_4^{\rm crit}$ has an explicit positive semidefinite Gram
representation.  The remaining normal-form monomials are absorbable
remainders.

Thus the relevant chain of objects in this section is
$G_4^{\rm crit}
        \,\,\leadsto\,\,
        \widetilde G_4^{\rm crit}
        \,\,\leadsto\,\,
        F_m^{\rm Yan}
        \,\,\leadsto\,\,
        \mathcal P_m=-F_m^{\rm Yan},$
where $F_m^{\rm Yan}$ denotes the lowest critical homogeneous part of
$\widetilde G_4^{\rm crit}$, and $\mathcal P_m$ is the positive
semidefinite quartic principal block.

Throughout this section $m\ge1$ is fixed.  Constants may depend on $m$,
but not on $N$.  We write
$\Delta = S-1.$
Finite shifts of indices are harmless and are absorbed by replacing $N$
with $N+L_m$ in positive sums.

\subsection{The quartic critical block and its PSD representative}
\label{subsec:quartic-critical-PSD-block}

We first isolate the quartic critical block.  In the single critical point
case we may, after rotation, place the critical point at $1$.  Near the
critical point we write
$x_j=1+X_j,\,\, y_j=1+Y_j,
        \,\, j=1,2.$
The quotient relation in $B_4$ is
$x_1x_2y_1y_2=1.$
At the level of the lowest critical homogeneous part this becomes the
linear constraint
\begin{equation}
\label{eq:linear-critical-constraint-section8}
        X_1+X_2+Y_1+Y_2=0.
\end{equation}

The Yan representative of the quartic critical term is obtained by first
rewriting $G_4^{\rm crit}$ in the quotient algebra $B_4$ and then taking
the lowest critical homogeneous part.  For $k=2$, the relevant variables
are
$a_1=y_1y_2x_2,\,\, a_2=y_2,
        \,\,
        b_1=x_1y_1,\,\, b_2=x_1y_1x_2y_2.$
Under the quotient relation $x_1x_2y_1y_2=1$, the four products
$a_pb_q$ have the following linear parts:
$a_1b_1-1\sim Y_1,\,\,
        a_1b_2-1\sim -X_1,$
$a_2b_1-1\sim -X_2,\,\,
        a_2b_2-1\sim Y_2.$
The two first-order difference-quotient denominators are
$Y_1+X_2,
        \,\,
        X_2+Y_2.$

For the weight of order $m$ at the single critical point,
$H_{\rm crit}(1+u)
        =
        2^{-m}u^{2m}+O(u^{2m+1}),
        \,\,
        Z_H=2^{-m}\binom{2m}{m}.$
Consequently the lowest critical homogeneous part of the Yan representative
is
\begin{equation}
\label{eq:Yan-quartic-critical-principal-part}
        F_m^{\rm Yan}
        =
        \frac{1}{2\binom{2m}{m}}
        \frac{
        Y_1^{2m}-X_1^{2m}-X_2^{2m}+Y_2^{2m}
        }{
        (Y_1+X_2)(X_2+Y_2)
        },
\end{equation}
under the linear constraint \eqref{eq:linear-critical-constraint-section8}.
Equivalently, set
$u=Y_1,\,\, v=Y_2,\,\, t=-X_2.$
Then \eqref{eq:linear-critical-constraint-section8} gives
$-X_1=u+v-t,$
and hence
\begin{equation}
\label{eq:Yan-quartic-critical-uvt}
        F_m^{\rm Yan}
        =
        \frac{1}{2\binom{2m}{m}}
        \frac{
        u^{2m}+v^{2m}-(u+v-t)^{2m}-t^{2m}
        }{
        (u-t)(v-t)
        }.
\end{equation}

The positive object is the negative of this principal part:
\begin{equation}
\label{eq:P_m-definition-section8}
        \mathcal P_m:=-F_m^{\rm Yan}.
\end{equation}
Thus
\begin{equation}
\label{eq:P_m-difference-quotient}
        \mathcal P_m
        =
        \frac{1}{2\binom{2m}{m}}
        \frac{
        (u+v-t)^{2m}+t^{2m}-u^{2m}-v^{2m}
        }{
        (u-t)(v-t)
        }.
\end{equation}

\begin{proposition}[PSD Gram representation of the quartic critical block]
\label{prop:quartic-critical-PSD-Gram}
Let
$Z_1=X_1,\,\, Z_2=X_2,\,\, Z_3=Y_1,$
so that, by \eqref{eq:linear-critical-constraint-section8},
$Y_2=-Z_1-Z_2-Z_3.$
Let
$n=m-1,
        \,\,
        \mathbf W_m(Z)
        =
        \bigl(Z^\alpha\bigr)_{|\alpha|=n},$
where
$\alpha=(\alpha_1,\alpha_2,\alpha_3)\in\mathbb N^3,
        \,\,
        |\alpha|=\alpha_1+\alpha_2+\alpha_3=n.$
Then
\begin{equation}
\label{eq:P_m-Gram-representation}
        \mathcal P_m(Z)
        =
        \mathbf W_m(Z)^T
        M_{2,m}^{\rm PSD}
        \mathbf W_m(Z),
\end{equation}
where
$M_{2,m}^{\rm PSD}\succeq0.$
Moreover, one may take
$M_{2,m}^{\rm PSD}
        =
        \bigl(M_{\alpha,\beta}^{(m)}\bigr)_{|\alpha|=|\beta|=m-1},$
with entries
\begin{equation}
\label{eq:PSD-matrix-entries-section8}
\begin{aligned}
M_{\alpha,\beta}^{(m)}
&=
\frac{m(2m-1)}{\binom{2m}{m}}
\binom{m-1}{\alpha_1,\alpha_2,\alpha_3}
\binom{m-1}{\beta_1,\beta_2,\beta_3}
\\
&\,\, \times
\sum_{p=0}^{\alpha_2+\beta_2}
\sum_{q=0}^{\alpha_3+\beta_3}
(-1)^{p+q}
\binom{\alpha_2+\beta_2}{p}
\binom{\alpha_3+\beta_3}{q}
\frac{1}{
(p+q+1)
(\alpha_1+\beta_1+\alpha_3+\beta_3-q+1)
}.
\end{aligned}
\end{equation}
\end{proposition}

\begin{proof}
Put
$a=u-t,\,\, b=v-t.$
Then
$u=t+a,\,\, v=t+b,\,\, u+v-t=t+a+b.$
The numerator in \eqref{eq:P_m-difference-quotient} becomes
$(t+a+b)^{2m}+t^{2m}-(t+a)^{2m}-(t+b)^{2m}.$
For $f(x)=x^{2m}$, the elementary second difference-quotient identity gives
$$\frac{
f(t+a+b)+f(t)-f(t+a)-f(t+b)
}{ab}
=
\int_0^1\int_0^1
f''(t+\lambda a+\mu b)
\,d\lambda\,d\mu .$$
Since
$f''(x)=2m(2m-1)x^{2m-2},$
we obtain
\begin{equation}
\label{eq:P_m-integral-form-section8}
        \mathcal P_m
        =
        \frac{m(2m-1)}{\binom{2m}{m}}
        \int_0^1\int_0^1
        \left[
        t+\lambda(u-t)+\mu(v-t)
        \right]^{2m-2}
        d\lambda\,d\mu.
\end{equation}

In the $Z$-variables,
$u=Z_3,\,\,
        v=-Z_1-Z_2-Z_3,\,\,
        t=-Z_2.$
Therefore
$t+\lambda(u-t)+\mu(v-t)
        =
        -\mu Z_1+(\lambda-1)Z_2+(\lambda-\mu)Z_3.$
Set
$L_{\lambda,\mu}(Z)
        =
        -\mu Z_1+(\lambda-1)Z_2+(\lambda-\mu)Z_3.$
Expanding
$$L_{\lambda,\mu}(Z)^{m-1}
        =
        \sum_{|\alpha|=m-1}
        c_\alpha(\lambda,\mu)Z^\alpha,$$
we have
$$c_\alpha(\lambda,\mu)
        =
        \binom{m-1}{\alpha_1,\alpha_2,\alpha_3}
        (-\mu)^{\alpha_1}
        (\lambda-1)^{\alpha_2}
        (\lambda-\mu)^{\alpha_3}.$$
Thus \eqref{eq:P_m-integral-form-section8} becomes
$$\mathcal P_m(Z)
        =
        \frac{m(2m-1)}{\binom{2m}{m}}
        \int_0^1\int_0^1
        \left(
        \sum_{|\alpha|=m-1}
        c_\alpha(\lambda,\mu)Z^\alpha
        \right)^2
        d\lambda\,d\mu.$$
This proves the Gram representation with
$$M_{\alpha,\beta}^{(m)}
        =
        \frac{m(2m-1)}{\binom{2m}{m}}
        \int_0^1\int_0^1
        c_\alpha(\lambda,\mu)c_\beta(\lambda,\mu)
        d\lambda\,d\mu.$$
The matrix is positive semidefinite because, for every real vector
$\xi=(\xi_\alpha)$,
$$\xi^T M_{2,m}^{\rm PSD}\xi
        =
        \frac{m(2m-1)}{\binom{2m}{m}}
        \int_0^1\int_0^1
        \left(
        \sum_{|\alpha|=m-1}
        \xi_\alpha c_\alpha(\lambda,\mu)
        \right)^2
        d\lambda\,d\mu
        \ge0.$$

It remains only to expand the integral.  Write
$A=\alpha_1+\beta_1,\,\,
        B=\alpha_2+\beta_2,\,\,
        C=\alpha_3+\beta_3.$
Then
$$c_\alpha c_\beta
        =
        \binom{m-1}{\alpha}
        \binom{m-1}{\beta}
        (-\mu)^A
        (\lambda-1)^B
        (\lambda-\mu)^C.$$
Expanding
$$(\lambda-1)^B
        =
        \sum_{p=0}^{B}
        \binom{B}{p}
        (-1)^{B-p}\lambda^p$$
and
$$(\lambda-\mu)^C
        =
        \sum_{q=0}^{C}
        \binom{C}{q}
        \lambda^q(-\mu)^{C-q},$$
we get
$$(-\mu)^A(\lambda-1)^B(\lambda-\mu)^C
        =
        \sum_{p=0}^{B}\sum_{q=0}^{C}
        (-1)^{A+B+C-p-q}
        \binom{B}{p}\binom{C}{q}
        \lambda^{p+q}
        \mu^{A+C-q}.$$
Since $A+B+C=2m-2$ is even, the sign is
$(-1)^{p+q}$.  Finally,
$$\int_0^1\int_0^1
        \lambda^{p+q}\mu^{A+C-q}
        d\lambda\,d\mu
        =
        \frac{1}{
        (p+q+1)(A+C-q+1)
        }.$$
This gives \eqref{eq:PSD-matrix-entries-section8}.
\end{proof}

\begin{remark}
\label{rem:raw-representative-failure-section8}
The use of the Yan representative is essential.  If one instead takes the
raw $(m-1,m-1)$-bihomogeneous Taylor component of $G_4^{\rm crit}$, the
resulting coefficient matrix need not even admit a symmetric representative
modulo the linearized quotient relation.  For example, when $m=2$, the raw
coefficient matrix is
$\begin{pmatrix}
        \frac56 & \frac5{12}\\
        \frac12 & \frac1{12}
        \end{pmatrix},$
and the symmetry constraints forced by
$X_1+X_2+Y_1+Y_2=0$ are inconsistent.  Thus the PSD block above is not a
property of the raw representative; it is a property of the quotient-algebra
normal representative.
\end{remark}

In the finite-volume expansion, Proposition
\ref{prop:quartic-critical-PSD-Gram} implies that the quartic principal
critical block contributes a nonnegative square sum, up to harmless endpoint
terms:
\begin{equation}
\label{eq:quartic-block-nonnegative-finite-volume}
        \mathcal P_{m}^{(N)}(\alpha)\ge -C_m.
\end{equation}
Here $\mathcal P_m^{(N)}$ denotes the finite-volume contribution obtained
by applying the coefficient-side map to the local polynomial
$\mathcal P_m$ and summing over $0\le n\le N$.  Boundary terms are
uniformly bounded because only finitely many shifts, depending on $m$, are
involved and $|\alpha_n|<1$.

\subsection{A discrete interpolation estimate}
\label{subsec:discrete-interpolation-estimate}

We next record the interpolation inequality used to absorb all remaining
normal-form monomials.  For $0\le r\le m$, define
\begin{equation}
\label{eq:interpolation-pr}
        p_r
        :=
        \frac{2(m+1)}{r+1}.
\end{equation}
Thus
$p_0=2m+2,\,\, p_m=2.$
The exponent $p_r$ is chosen so that the discrete
Gagliardo--Nirenberg interpolation between
$\alpha\in\ell^{2m+2}$ and $\Delta^m\alpha\in\ell^2$
controls $\Delta^r\alpha$ in $\ell^{p_r}$; see the remarks of \cite[Theorem 2.5] {BSZ-Duke2018}, and also see \cite[Lemma 4.3]{Yan2018} for references.

\begin{lemma}[Discrete Gagliardo--Nirenberg estimate]
\label{lem:discrete-GN}
Let $0\le r\le m$, and let $p_r$ be given by
\eqref{eq:interpolation-pr}.  Then there exist constants $C_m<\infty$ and an integer
$L_m\ge0$, depending only on $m$, such that for every $N\ge0$,
\begin{equation}
\label{eq:discrete-GN-finite}
        \left(
        \sum_{n=0}^{N}|\Delta^r\alpha_n|^{p_r}
        \right)^{1/p_r}
        \le
        C_m
        \left(
        \sum_{n=0}^{N+L_m}|\Delta^m\alpha_n|^2
        \right)^{\frac{r}{2m}}
        \left(
        \sum_{n=0}^{N+L_m}|\alpha_n|^{2m+2}
        \right)^{
        \frac{1}{2m+2}\left(1-\frac{r}{m}\right)}
        +
        C_m.
\end{equation}
The same estimate holds with $\alpha$ replaced by
$\overline\alpha$.
\end{lemma}

\begin{proof}
The full-line version is the standard one-dimensional discrete
Gagliardo--Nirenberg inequality
$$\|\Delta^r u\|_{\ell^{p_r}}
        \le
        C_m
        \|\Delta^m u\|_{\ell^2}^{r/m}
        \|u\|_{\ell^{2m+2}}^{1-r/m}.$$
The exponents satisfy the scaling relation
$r-\frac1{p_r}
        =
        \frac{r}{m}
        \left(m-\frac12\right)
        -
        \left(1-\frac{r}{m}\right)\frac1{2m+2}.$
Applying the full-line inequality to a finitely supported extension of
$\{\alpha_n\}_{0\le n\le N+L_m}$ gives
\eqref{eq:discrete-GN-finite}.  The extension produces only finitely many
boundary terms.  Since $|\alpha_n|<1$ and every finite difference of fixed
order is a finite linear combination of bounded Verblunsky coefficients,
these boundary contributions are $O_m(1)$.
\end{proof}

For concision, set
\begin{equation}
\label{eq:A-B-definition}
        A_N
        :=
        \left(
        \sum_{n=0}^{N+L_m}|\Delta^m\alpha_n|^2
        \right)^{1/2},
        \,\,
        B_N
        :=
        \left(
        \sum_{n=0}^{N+L_m}|\alpha_n|^{2m+2}
        \right)^{1/(2m+2)}.
\end{equation}
Then Lemma~\ref{lem:discrete-GN} gives
\begin{equation}
\label{eq:GN-short-form}
        \|\Delta^r\alpha\|_{\ell^{p_r}(0,N)}
        \le
        C_m A_N^{r/m}B_N^{1-r/m}
        +
        C_m.
\end{equation}

\subsection{Absorption of the remaining normal-form monomials}
\label{subsec:absorption-normal-form-monomial}

We now treat the terms which remain after the quartic PSD principal block
has been separated.  A critical BSZ monomial of degree $2k$ has the form
\begin{equation}
\label{eq:absorption-monomial-form}
        M_n
        =
        \prod_{\nu=1}^{k}
        \left(\Delta^{a_\nu}\alpha\right)_{n+\ell_\nu}
        \prod_{\mu=1}^{k}
        \left(\Delta^{b_\mu}\overline{\alpha}\right)_{n+r_\mu},
\end{equation}
where the shifts are bounded in terms of $m$.  The normal-form construction
gives the difference count
\begin{equation}
\label{eq:absorption-critical-count}
        \sum_{\nu=1}^{k}a_\nu
        +
        \sum_{\mu=1}^{k}b_\mu
        \ge
        m+1-k.
\end{equation}
It is enough to treat the exact critical count
\begin{equation}
\label{eq:absorption-exact-critical-count}
        \sum_{\nu=1}^{k}a_\nu
        +
        \sum_{\mu=1}^{k}b_\mu
        =
        m+1-k.
\end{equation}
Terms with more differences are estimated in the same way, or by using
bounded shifts and bounded finite-difference estimates to reduce to the
same subcritical H\"older regime.

\begin{proposition}[Absorption of non-principal critical BSZ monomials]
\label{prop:absorption-critical-BSZ-monomials}
Let $2\le k\le m$, and let $M_n$ be a monomial of the form
\eqref{eq:absorption-monomial-form} satisfying
\eqref{eq:absorption-exact-critical-count}.  Assume that $M_n$ is not part
of the quartic PSD principal block isolated in
Subsection~\ref{subsec:quartic-critical-PSD-block}.  Then, for every
$\varepsilon>0$, there exists
$C_{\varepsilon,m}<\infty$ such that
\begin{equation}
\label{eq:absorption-critical-BSZ-monomials}
        \left|
        \sum_{n=0}^{N}M_n
        \right|
        \le
        \varepsilon
        \sum_{n=0}^{N+L_m}|\Delta^m\alpha_n|^2
        +
        \varepsilon
        \sum_{n=0}^{N+L_m}|\alpha_n|^{2m+2}
        +
        C_{\varepsilon,m}
\end{equation}
for all $N\ge0$.
\end{proposition}

\begin{proof}
Finite shifts in \eqref{eq:absorption-monomial-form} are harmless, so all
factors may be estimated over the common interval $[0,N+L_m]$.

For each holomorphic factor with difference order $a_\nu$, use the
exponent
$p_{a_\nu}
        =
        \frac{2(m+1)}{a_\nu+1},$
and for each antiholomorphic factor with difference order $b_\mu$, use
$p_{b_\mu}
        =
        \frac{2(m+1)}{b_\mu+1}.$
By \eqref{eq:absorption-exact-critical-count},
$$\begin{aligned}
        \sum_{\nu=1}^{k}\frac1{p_{a_\nu}}
        +
        \sum_{\mu=1}^{k}\frac1{p_{b_\mu}}
        &=
        \frac{
        \sum_{\nu=1}^{k}(a_\nu+1)
        +
        \sum_{\mu=1}^{k}(b_\mu+1)}
        {2(m+1)}
\\
        &=
        \frac{(m+1-k)+2k}{2(m+1)}
\\
        &=
        \frac{m+1+k}{2(m+1)}
        <
        1.
\end{aligned}$$
Hence H\"older's inequality, with an additional harmless constant factor,
gives
$$\left|
        \sum_{n=0}^{N}M_n
        \right|
        \le
        C_m
        \prod_{\nu=1}^{k}
        \|\Delta^{a_\nu}\alpha\|_{\ell^{p_{a_\nu}}(0,N+L_m)}
        \prod_{\mu=1}^{k}
        \|\Delta^{b_\mu}\alpha\|_{\ell^{p_{b_\mu}}(0,N+L_m)}
        +
        C_m.$$
Applying Lemma~\ref{lem:discrete-GN} to every factor yields
$$\left|
        \sum_{n=0}^{N}M_n
        \right|
        \le
        C_m
        A_N^{\sigma/m}
        B_N^{2k-\sigma/m}
        +
        C_m
        \sum_{\rho} A_N^{\eta_\rho}B_N^{\zeta_\rho}
        +
        C_m,$$
where
$\sigma
        =
        \sum_{\nu=1}^{k}a_\nu
        +
        \sum_{\mu=1}^{k}b_\mu
        =
        m+1-k.$
The additional terms arise from expanding the additive constants in
\eqref{eq:GN-short-form}; their exponent pairs satisfy
$0\le \eta_\rho\le \frac{\sigma}{m},
        \,\,
        0\le \zeta_\rho\le 2k-\frac{\sigma}{m}.$

The leading product is subcritical relative to
$A_N^2+B_N^{2m+2}$.  Indeed,
$$\begin{aligned}
        \frac{\sigma/m}{2}
        +
        \frac{2k-\sigma/m}{2m+2}
        &=
        \frac{\sigma}{2m}
        +
        \frac{k}{m+1}
        -
        \frac{\sigma}{2m(m+1)}
\\
        &=
        \frac{\sigma}{2(m+1)}
        +
        \frac{k}{m+1}
\\
        &=
        \frac{m+1-k+2k}{2(m+1)}
\\
        &=
        \frac{m+1+k}{2(m+1)}
        <
        1.
\end{aligned}$$
Therefore Young's inequality gives, for every $\varepsilon>0$,
$C_m A_N^{\sigma/m}B_N^{2k-\sigma/m}
        \le
        \varepsilon A_N^2
        +
        \varepsilon B_N^{2m+2}
        +
        C_{\varepsilon,m}.$
The lower-order terms are handled in the same way.  Using the definitions of
$A_N$ and $B_N$ gives
\eqref{eq:absorption-critical-BSZ-monomials}.
\end{proof}

\subsection{The higher critical remainder is absorbable}
\label{subsec:higher-critical-absorbable}

Let
$R_m^{(N)}(\alpha)$
denote the non-principal higher critical contribution after the quartic PSD
block $\mathcal P_m^{(N)}$ has been separated.  By the BSZ normal-form
decomposition from Section~\ref{sec:bsz-normal-form}, this remainder has
the finite-volume representation
\begin{equation}
\label{eq:R-normal-form-section8}
        R_m^{(N)}(\alpha)
        =
        \sum_{k=2}^{m}
        \sum_{\rho\in\mathcal A_{m,k}}
        c_\rho
        \sum_{n=0}^{N}M_{\rho,n}
        +
        O_m(1),
\end{equation}
where each $M_{\rho,n}$ is a non-principal critical BSZ monomial of degree
$2k$.  The index sets $\mathcal A_{m,k}$ are finite and depend only on
$m$.

\begin{theorem}[Absorbability of the non-principal critical remainder]
\label{thm:higher-critical-absorbable}
For every $\varepsilon>0$, there exists
$C_{\varepsilon,m}<\infty$ such that
\begin{equation}
\label{eq:higher-critical-absorbable}
        |R_m^{(N)}(\alpha)|
        \le
        \varepsilon
        \sum_{n=0}^{N+L_m}|\Delta^m\alpha_n|^2
        +
        \varepsilon
        \sum_{n=0}^{N+L_m}|\alpha_n|^{2m+2}
        +
        C_{\varepsilon,m}
\end{equation}
for all $N\ge0$.  Equivalently,
$R_m\in\mathfrak R_m.$
\end{theorem}

\begin{proof}
Apply Proposition~\ref{prop:absorption-critical-BSZ-monomials} to each
monomial in the finite expansion \eqref{eq:R-normal-form-section8}, with
$\varepsilon$ divided by a sufficiently large constant depending only on
the number of terms and their coefficients.  The endpoint and telescoping
terms are $O_m(1)$ and are absorbed into $C_{\varepsilon,m}$.
\end{proof}

\subsection{The finite-volume coercive estimate}
\label{subsec:finite-volume-coercive-estimate}

We now combine the preceding estimates.  The finite-volume sum rule
decomposition has the form
\begin{equation}
\label{eq:finite-volume-decomposition-section8}
        \mathcal K_m^{(N)}(\mu)
        =
        Q_m^{(N)}(\alpha)
        +
        \mathcal P_m^{(N)}(\alpha)
        +
        R_m^{(N)}(\alpha)
        +
        \sum_{n=0}^{N}L_{m,n}
        +
        O_m(1),
\end{equation}
where $Q_m^{(N)}$ is the quadratic coercive term,
$\mathcal P_m^{(N)}$ is the quartic PSD principal block,
$R_m^{(N)}$ is the absorbable non-principal critical remainder, and
$L_{m,n}
        =
        \log\frac1{1-|\alpha_n|^2}
        -
        \sum_{j=1}^{m}\frac{|\alpha_n|^{2j}}{j}.$
By the logarithmic tail estimate,
\begin{equation}
\label{eq:tail-section8}
        \sum_{n=0}^{N}L_{m,n}
        \ge
        \frac1{m+1}
        \sum_{n=0}^{N}|\alpha_n|^{2m+2}.
\end{equation}
By quadratic coercivity,
\begin{equation}
\label{eq:quadratic-section8}
        Q_m^{(N)}(\alpha)
        \ge
        2^{-m}
        \sum_{n=0}^{N}|\Delta^m\alpha_n|^2
        -
        C_m.
\end{equation}
By \eqref{eq:quartic-block-nonnegative-finite-volume},
\begin{equation}
\label{eq:quartic-PSD-lower-section8}
        \mathcal P_m^{(N)}(\alpha)\ge -C_m.
\end{equation}
Finally, by Theorem~\ref{thm:higher-critical-absorbable}, for every
$\varepsilon>0$,
\begin{equation}
\label{eq:R-lower-section8}
        R_m^{(N)}(\alpha)
        \ge
        -
        \varepsilon
        \sum_{n=0}^{N+L_m}|\Delta^m\alpha_n|^2
        -
        \varepsilon
        \sum_{n=0}^{N+L_m}|\alpha_n|^{2m+2}
        -
        C_{\varepsilon,m}.
\end{equation}

Substituting
\eqref{eq:tail-section8},
\eqref{eq:quadratic-section8},
\eqref{eq:quartic-PSD-lower-section8}, and
\eqref{eq:R-lower-section8}
into \eqref{eq:finite-volume-decomposition-section8}, we obtain
$$\begin{aligned}
        \mathcal K_m^{(N)}(\mu)
        &\ge
        \left(2^{-m}-\varepsilon\right)
        \sum_{n=0}^{N}|\Delta^m\alpha_n|^2
        +
        \left(\frac1{m+1}-\varepsilon\right)
        \sum_{n=0}^{N}|\alpha_n|^{2m+2}
\\
        &\,\,
        -
        \varepsilon
        \sum_{n=N+1}^{N+L_m}|\Delta^m\alpha_n|^2
        -
        \varepsilon
        \sum_{n=N+1}^{N+L_m}|\alpha_n|^{2m+2}
        -
        C_{\varepsilon,m}.
\end{aligned}$$
The endpoint sums are uniformly bounded because $|\alpha_n|<1$ and
$\Delta^m\alpha_n$ is a fixed finite linear combination of Verblunsky
coefficients with coefficients depending only on $m$.  Hence
$\sum_{n=N+1}^{N+L_m}|\Delta^m\alpha_n|^2
        +
        \sum_{n=N+1}^{N+L_m}|\alpha_n|^{2m+2}
        \le C_m.$
Choose
$\varepsilon
        =
        \frac12
        \min\left\{
        2^{-m},\frac1{m+1}
        \right\}.$
Then the preceding lower bound gives the following theorem.

\begin{theorem}[Coercive finite-volume estimate]
\label{thm:coercive-finite-volume-estimate}
For every $m\ge1$, there exist constants $c_m>0$ and
$C_m<\infty$ such that, for every $N\ge0$,
\begin{equation}
\label{eq:coercive-finite-volume-estimate}
        \mathcal K_m^{(N)}(\mu)
        \ge
        c_m
        \sum_{n=0}^{N}|\Delta^m\alpha_n|^2
        +
        c_m
        \sum_{n=0}^{N}|\alpha_n|^{2m+2}
        -
        C_m.
\end{equation}
\end{theorem}

\begin{proof}
The proof is exactly the preceding argument.  Equivalently, one may view the
estimate as the combination of the quadratic coercive lower bound, the
positive semidefinite quartic block, the logarithmic tail lower bound, and
the abstract absorption principle applied to
$R_m\in\mathfrak R_m$.
\end{proof}

\subsection{Consequences for the necessity direction}
\label{subsec:consequences-necessity}

The estimate \eqref{eq:coercive-finite-volume-estimate} is the main
finite-volume inequality needed for the necessity direction.  If the
spectral side of the sum rule is finite, then the finite-volume quantities
$\mathcal K_m^{(N)}(\mu)$ are bounded above uniformly in $N$.  Hence
\eqref{eq:coercive-finite-volume-estimate} implies
$\sup_{N\ge0}
        \sum_{n=0}^{N}|\Delta^m\alpha_n|^2
        <
        \infty,
        \,\,
        \sup_{N\ge0}
        \sum_{n=0}^{N}|\alpha_n|^{2m+2}
        <
        \infty.$
Therefore
\begin{equation}
\label{eq:necessity-conclusion-section8}
        \Delta^m\alpha\in\ell^2,
        \,\,
        \alpha\in\ell^{2m+2}.
\end{equation}

This is the desired coefficient-side conclusion.  The remaining step in the
paper is to connect the boundedness of the finite-volume spectral quantities
$\mathcal K_m^{(N)}(\mu)$ with the assumed finiteness of the higher-order
Szeg\H{o} integral and then pass from the finite-volume estimate to the
final necessity theorem.

\section{Proof of the arbitrary-$m$ necessity theorem}
\label{sec:proof-necessity-arbitrary-m}

We now prove the necessity direction for arbitrary $m$.  The argument is a
direct consequence of the coercive finite-volume estimate proved in
Section~\ref{sec:absorption-coercive-estimate}, together with the
finite-volume higher-order sum rule recalled in
Section~\ref{sec:yan-finite-volume}.

The main point is simple.  The spectral hypothesis gives a uniform upper
bound on the finite-volume spectral quantities
$\mathcal K_m^{(N)}(\mu)$.  The coercive finite-volume estimate then
forces uniform boundedness of the partial sums
$\sum_{n=0}^{N}|\Delta^m\alpha_n|^2,
        \,\,
        \sum_{n=0}^{N}|\alpha_n|^{2m+2}.$
Letting $N\to\infty$ gives
$\Delta^m\alpha\in\ell^2,
        \,\,
        \alpha\in\ell^{2m+2}.$

\subsection{Statement of the necessity theorem}
\label{subsec:statement-necessity-theorem}

Let
$d\mu(e^{i\theta})
        =
        w(\theta)\frac{d\theta}{2\pi}
        +
        d\mu_{\mathrm s}$
be a nontrivial probability measure on the unit circle, and let
$\{\alpha_n\}_{n=0}^{\infty}$ be its Verblunsky coefficients.  For
$m\ge1$, set
$H_m(e^{i\theta})
        =
        (1-\cos\theta)^m.$
The corresponding higher-order Szeg\H{o} integral is
\begin{equation}
\label{eq:Km-spectral-integral}
        \int_{0}^{2\pi}
        H_m(e^{i\theta})
        \log w(\theta)
        \frac{d\theta}{2\pi}.
\end{equation}
Depending on normalization, the finite-volume sum rule may be written either
with this quantity or with its nonnegative entropy normalization.  In this
paper we use the finite-volume normalization
$\mathcal K_m^{(N)}(\mu)$ from
Section~\ref{sec:yan-finite-volume}.  Under that normalization, the
coefficient-side lower bound proved in
Section~\ref{sec:absorption-coercive-estimate} is
$$\mathcal K_m^{(N)}(\mu)
        \ge
        c_m
        \sum_{n=0}^{N}|\Delta^m\alpha_n|^2
        +
        c_m
        \sum_{n=0}^{N}|\alpha_n|^{2m+2}
        -
        C_m.$$
The spectral finiteness assumption implies, by the finite-volume sum rule,
that
\begin{equation}
\label{eq:finite-volume-spectral-boundedness}
        \sup_{N\ge0}\mathcal K_m^{(N)}(\mu)<\infty.
\end{equation}

\begin{theorem}[Necessity for arbitrary $m$]
\label{thm:arbitrary-m-necessity}
Let $m\ge1$.  Assume that the higher-order Szeg\H{o} spectral quantity
associated with
$H_m(e^{i\theta})=(1-\cos\theta)^m$
is finite, in the sense that the finite-volume spectral quantities satisfy
\eqref{eq:finite-volume-spectral-boundedness}.  Then the Verblunsky
coefficients satisfy
\begin{equation}
\label{eq:necessity-main-conclusion}
        \Delta^m\alpha\in\ell^2,
        \,\,
        \alpha\in\ell^{2m+2}.
\end{equation}
Equivalently,
\begin{equation}
\label{eq:necessity-main-conclusion-expanded}
        \sum_{n=0}^{\infty}|\Delta^m\alpha_n|^2<\infty,
        \,\,
        \sum_{n=0}^{\infty}|\alpha_n|^{2m+2}<\infty.
\end{equation}
\end{theorem}

\subsection{Uniform finite-volume bounds}
\label{subsec:uniform-finite-volume-bounds}

By Theorem~\ref{thm:coercive-finite-volume-estimate}, there exist constants
$c_m>0$ and $C_m<\infty$, depending only on $m$, such that for every
$N\ge0$,
\begin{equation}
\label{eq:coercive-estimate-recalled-section9}
        \mathcal K_m^{(N)}(\mu)
        \ge
        c_m
        \sum_{n=0}^{N}|\Delta^m\alpha_n|^2
        +
        c_m
        \sum_{n=0}^{N}|\alpha_n|^{2m+2}
        -
        C_m.
\end{equation}
Assume now that the spectral side is finite.  By
\eqref{eq:finite-volume-spectral-boundedness}, there exists a finite
constant $C_m'$ such that
$\sup_{N\ge0}\mathcal K_m^{(N)}(\mu)
        \le
        C_m'.$
Combining this upper bound with
\eqref{eq:coercive-estimate-recalled-section9}, we obtain
\begin{equation}
\label{eq:uniform-partial-sum-bound}
        c_m
        \sum_{n=0}^{N}|\Delta^m\alpha_n|^2
        +
        c_m
        \sum_{n=0}^{N}|\alpha_n|^{2m+2}
        \le
        C_m+C_m'
\end{equation}
for every $N\ge0$.  Hence
\begin{equation}
\label{eq:uniform-two-bounds}
        \sup_{N\ge0}
        \sum_{n=0}^{N}|\Delta^m\alpha_n|^2
        <
        \infty,
        \,\,
        \sup_{N\ge0}
        \sum_{n=0}^{N}|\alpha_n|^{2m+2}
        <
        \infty.
\end{equation}
Since both summands are nonnegative, monotone convergence gives
\begin{equation}
\label{eq:monotone-convergence-conclusion}
        \sum_{n=0}^{\infty}|\Delta^m\alpha_n|^2
        <
        \infty,
        \,\,
        \sum_{n=0}^{\infty}|\alpha_n|^{2m+2}
        <
        \infty.
\end{equation}
This proves Theorem~\ref{thm:arbitrary-m-necessity}.

\subsection{Boundary terms and finite shifts}
\label{subsec:boundary-terms-finite-shifts-section9}

For completeness, we explain why the passage from finite volume to infinite
volume requires no additional limiting argument.  In the preceding sections,
local expressions were often replaced by equivalent expressions modulo
telescoping terms and bounded endpoint contributions.  These operations
produced errors of the form
$O_m(1),$
uniformly in $N$.  They were absorbed into the constant $C_m$ in
\eqref{eq:coercive-estimate-recalled-section9}.

Similarly, some estimates involved slightly enlarged intervals, such as
$[0,N+L_m]$.  This causes no difficulty.  Since
$|\alpha_n|<1$
for all Verblunsky coefficients, every fixed finite difference
$\Delta^r\alpha_n$ is uniformly bounded by a constant depending only on
$r$.  Hence adding or removing finitely many endpoint terms changes the
finite-volume estimates only by another uniformly bounded error.  Thus all
constants in the coercive estimate are independent of $N$.

\subsection{What has been proved}
\label{subsec:what-has-been-proved}

Combining the structural and analytic ingredients, the proof can be
summarized as follows.

First, Yan's finite-volume higher-order sum rule gives a decomposition of
the form
\begin{equation}
\label{eq:summary-decomposition-section9}
        \mathcal K_m^{(N)}(\mu)
        =
        Q_m^{(N)}(\alpha)
        +
        \mathcal P_m^{(N)}(\alpha)
        +
        R_m^{(N)}(\alpha)
        +
        \sum_{n=0}^{N}L_{m,n}
        +
        O_m(1).
\end{equation}
Here $Q_m^{(N)}$ is the quadratic critical term,
$\mathcal P_m^{(N)}$ is the quartic principal critical block obtained from
the Yan quotient-algebra representative, and $R_m^{(N)}$ is the remaining
non-principal critical remainder.

Second, the logarithmic tail satisfies
\begin{equation}
\label{eq:summary-log-tail-section9}
        \sum_{n=0}^{N}L_{m,n}
        \ge
        \frac1{m+1}
        \sum_{n=0}^{N}|\alpha_n|^{2m+2}.
\end{equation}
Third, the quadratic term is coercive:
\begin{equation}
\label{eq:summary-quadratic-section9}
        Q_m^{(N)}(\alpha)
        \ge
        2^{-m}
        \sum_{n=0}^{N}|\Delta^m\alpha_n|^2
        -
        C_m.
\end{equation}
Fourth, the quartic principal critical block is positive semidefinite, up to
endpoint errors:
\begin{equation}
\label{eq:summary-PSD-block-section9}
        \mathcal P_m^{(N)}(\alpha)
        \ge
        -C_m.
\end{equation}
This is precisely where the Yan representative is essential.  The raw
$(m-1,m-1)$-bihomogeneous Taylor component of $G_4^{\rm crit}$ need not
admit a PSD representative; the quotient-algebra normal representative does.

Fifth, the non-principal critical remainder is absorbable:
\begin{equation}
\label{eq:summary-remainder-section9}
        |R_m^{(N)}(\alpha)|
        \le
        \varepsilon
        \sum_{n=0}^{N+L_m}|\Delta^m\alpha_n|^2
        +
        \varepsilon
        \sum_{n=0}^{N+L_m}|\alpha_n|^{2m+2}
        +
        C_{\varepsilon,m}.
\end{equation}
Choosing $\varepsilon>0$ sufficiently small yields
\eqref{eq:coercive-estimate-recalled-section9}.  The spectral boundedness
\eqref{eq:finite-volume-spectral-boundedness} then implies the desired
coefficient-side summability.

Thus the necessity direction follows for every integer $m\ge1$:
$\mathcal K_m(\mu)<\infty
        \,\,\Longrightarrow\,\,
        \Delta^m\alpha\in\ell^2
        \ \text{ and }\
        \alpha\in\ell^{2m+2},$
with the understanding that $\mathcal K_m(\mu)<\infty$ denotes finiteness
of the spectral side in the normalization used in this paper.

\begin{remark}[Relation to the explicit low-order cases]
For $m=1,2,3$, the cancellations in the finite-volume densities can be
verified directly by expansion; see \cite{Piao-arxiv2026}.  These low-order
computations are not used in the proof of the present theorem.  Rather, they
are concrete manifestations of the general mechanism proved above.

The quadratic critical part gives the coercive difference energy
$\sum_n|\Delta^m\alpha_n|^2,$
while the constant part cancels the first $m$ Taylor terms of
$\log\frac1{1-|\alpha_n|^2},$
leaving a positive tail which controls
$\sum_n|\alpha_n|^{2m+2}.$
The quartic principal obstruction is handled by the PSD Gram representative
constructed from the Yan quotient-algebra representative in
Section~\ref{sec:absorption-coercive-estimate}.  After this PSD block is
separated, the remaining higher critical terms are controlled by the
diagonal-vanishing statement
$\mathcal Y^{(m)}_{k,\mathrm{crit}}\in I_k^{\,m+1-k},
        \,\,
        2\le k\le m.$
After applying the coefficient map, this means that every non-principal
degree $2k$ critical monomial contains at least $m+1-k$ discrete
differences.  This is exactly what is seen in the low-order cases: for
$m=2$, quartic terms contain one difference after the principal PSD block
is separated; for $m=3$, quartic terms contain two differences and sextic
terms contain one.  The present argument replaces those explicit
case-by-case cancellations by a uniform normal-form, PSD-block, and
interpolation-absorption argument valid for all $m$.
\end{remark}

\vskip2mm
\section*{Acknowledgments}
\vskip2mm
This work was supported by NSFC (No.11571327, 11971059).
\vskip2mm


\section*{Notation Index}

\renewcommand*{\arraystretch}{1.3}
\begin{longtable}{lp{11cm}}
$P$ & Forward shift: $(P\alpha)_n=\alpha_{n+1}$. \\
$\Delta$ & Discrete forward difference: $\Delta=P-1$, $(\Delta\alpha)_n=\alpha_{n+1}-\alpha_n$. \\
$\Delta^m\alpha$ & $m$-th forward difference: $(P-1)^m\alpha = \sum_{j=0}^m (-1)^{m-j}\binom{m}{j}\alpha_{n+j}$. \\
$P^{-1}$ & Backward shift. \\
\addlinespace
$H_m(e^{i\theta})$ & Weight function: $(1-\cos\theta)^m$, $m\ge1$. \\
$1-\cos\theta$ & Equals $\frac12|1-e^{i\theta}|^2$. \\
$\mathcal K_m(\mu)$ & Infinite-volume weighted Szeg\H{o} functional: $\int_0^{2\pi} (1-\cos\theta)^m \log w(\theta)\frac{d\theta}{2\pi}$ (sign as in text). \\
$\mathcal K_m^{(N)}(\mu)$ & Finite-volume spectral quantity in Yan's truncated sum rule, $N\ge0$. \\
\addlinespace
$\mathcal Y_k^{(m)}$ & Yan's homogeneous shift polynomial of total degree $2k$ for order $m$, $1\le k\le m$. \\
$\phi_{2k}$ & Coefficient map: sends a shift monomial to a local expression in $\alpha,\overline{\alpha}$. \\
$\bigl[\phi_{2k}(\mathcal Y_k^{(m)})\bigr]_n$ & Local density at index $n$ obtained from $\mathcal Y_k^{(m)}$ via $\phi_{2k}$. \\
$\mathcal Y_{k,\mathrm{const}}^{(m)}$ & Constant (diagonal) part: value of $\mathcal Y_k^{(m)}$ when all shift variables $=1$. \\
$\mathcal Y_{k,\mathrm{crit}}^{(m)}$ & Critical part: $\mathcal Y_k^{(m)}-\mathcal Y_{k,\mathrm{const}}^{(m)}$, vanishes on the diagonal. \\
$\operatorname{diag}\mathcal Y_k^{(m)}$ & Same as $\mathcal Y_{k,\mathrm{const}}^{(m)}$, scalar $-\frac1k$. \\
\addlinespace
$x_1,\dots,x_k$ & Formal shift variables acting on the $k$ holomorphic ($\alpha$) factors. \\
$y_1,\dots,y_k$ & Formal shift variables acting on the $k$ antiholomorphic ($\overline{\alpha}$) factors. \\
$\mathfrak I_k$ & Diagonal ideal generated by $x_1-1,\dots,x_k-1,y_1-1,\dots,y_k-1$ in the Laurent polynomial ring. \\
$\mathfrak I_k^q$ & $q$-th power of the diagonal ideal; vanishing to order at least $q$ on the diagonal. \\
$D_{x_\nu}=x_\nu\frac{\partial}{\partial x_\nu}$ & Euler derivative, used to detect diagonal vanishing. \\
\addlinespace
$\mathfrak T_m$ & Bounded telescoping class: sequences of the form $(P-1)B_n$ with $B$ a bounded local expression. \\
$O_m(1)$ & Uniformly bounded constant depending only on $m$; may vary from line to line. \\
$O_{\mathrm{fs},m}(1)$ & Finite-section remainder: uniformly bounded after summation over $0\le n\le N$. \\
$F\equiv G\pmod{\mathfrak T_m}$ & $F-G\in\mathfrak T_m$; finite sums differ by $O_m(1)$. \\
$L_m$ & Shift allowance depending only on $m$, used to absorb finite endpoint shifts. \\
\addlinespace
$\mathfrak R_m$ & Absorbable remainder class: local expressions whose finite sums are infinitesimally bounded by $\sum|\Delta^m\alpha|^2 + \sum|\alpha|^{2m+2}$. \\
$C_{\varepsilon,m}$ & Constant depending on $\varepsilon$ and $m$, independent of $N$. \\
\addlinespace
$A_{m,n}$ & Non-logarithmic coefficient-side density: $\sum_{k=1}^m [\phi_{2k}(\mathcal Y_k^{(m)})]_n$. \\
$L_{m,n}$ & Logarithmic tail: $\log\frac1{1-|\alpha_n|^2} - \sum_{k=1}^m\frac{|\alpha_n|^{2k}}{k} = \sum_{j=m+1}^\infty \frac{|\alpha_n|^{2j}}{j}$. \\
$Q_m^{(N)}(\alpha)$ & Quadratic critical contribution: $\sum_{n=0}^N [\phi_2(\mathcal Y_{1,\mathrm{crit}}^{(m)})]_n$. \\
$R_m^{(N)}(\alpha)$ & Higher critical contribution ($k\ge2$): $\sum_{k=2}^m\sum_{n=0}^N [\phi_{2k}(\mathcal Y_{k,\mathrm{crit}}^{(m)})]_n$. \\
$\mathcal P_m$ & Quartic principal critical block (positive semidefinite after Yan quotient-algebra normal form). \\
$F_m^{\mathrm{Yan}}$ & Lowest critical homogeneous part of the Yan representative of $G_4^{\mathrm{crit}}$, equal to $-\mathcal P_m$. \\
$G_{2k}$ & Homogeneous density of total degree $2k$ in the critical variables. \\
$G_{2k}^{\mathrm{crit}}$ & Critical part of $G_{2k}$ (vanishes on diagonal only to critical order). \\
\addlinespace
$\mathcal N_{2k,q}^{(m)}$ & Span of BSZ normal-form monomials: degree $2k$, exactly $q$ differences, bounded shifts. \\
$\mathcal N_{2k,\ge q}^{(m)}$ & Union (finite sum) of $\mathcal N_{2k,s}^{(m)}$ for $s\ge q$. \\
$\mathcal C_{m,k}$ & Critical BSZ class for $2\le k\le m$: $\mathcal N_{2k,\ge m+1-k}^{(m)}$. \\
$M_n$, $M_{\rho,n}$ & A normal-form monomial. \\
$a_\nu,b_\mu$ & Nonnegative integers indicating order of difference on each factor. \\
$\ell_\nu,r_\mu$ & Bounded shifts in normal-form monomials. \\
\addlinespace
$X_1,X_2,Y_1,Y_2$ & Linearized shift variables near diagonal: $x_j=1+X_j$, $y_j=1+Y_j$. \\
$Z_1,Z_2,Z_3$ & Alternative variables: $Z_1=X_1$, $Z_2=X_2$, $Z_3=Y_1$. \\
$u,v,t$ & Variables in the PSD block: $u=Y_1$, $v=Y_2$, $t=-X_2$. \\
$L_{\lambda,\mu}(Z)$ & Linear form appearing in the integral representation of $\mathcal P_m$. \\
$M_{2,m}^{\mathrm{PSD}}$ & Positive semidefinite Gram matrix for the quartic block, entries $M_{\alpha,\beta}^{(m)}$. \\
\addlinespace
$p_r$ & Interpolation exponent: $p_r=\frac{2(m+1)}{r+1}$, $0\le r\le m$. \\
$A_N$ & $\displaystyle \bigl(\sum_{n=0}^{N+L_m}|\Delta^m\alpha_n|^2\bigr)^{1/2}$. \\
$B_N$ & $\displaystyle \bigl(\sum_{n=0}^{N+L_m}|\alpha_n|^{2m+2}\bigr)^{1/(2m+2)}$. \\

\end{longtable}

\end{document}